\documentclass{article}
\usepackage[utf8]{inputenc}
\usepackage[margin=1in]{geometry}
\usepackage{tikz-cd,tikz,caption,appendix}
\usetikzlibrary{intersections,positioning,calc,patterns,decorations.markings,arrows}
\usepackage{amsmath,amssymb,amsthm,wrapfig,enumerate,mathtools,cases,tcolorbox}
\usepackage[pdfencoding=auto,hidelinks]{hyperref}

\DeclareMathOperator{\C}{\mathbb{C}}
\DeclareMathOperator{\R}{\mathbb{R}}
\DeclareMathOperator{\Z}{\mathbb{Z}}

\DeclareMathOperator{\CP}{\mathbb{CP}}

\DeclareMathOperator{\Hom}{Hom}

\DeclareMathOperator{\Ad}{Ad}

\DeclareMathOperator{\Reom}{Re\,\Omega}
\DeclareMathOperator{\Imom}{Im\,\Omega}
\DeclareMathOperator{\SU}{SU}
\DeclareMathOperator{\su}{\mathfrak{su}}

\newcommand{\G}{\SU(2)^2\times \textnormal{U}(1)}

\DeclareMathOperator{\Sym}{Sym}

\newcommand{\Aab}{A^{\text{ab}}}

\newcommand{\norm}[1]{\vert #1 \vert}

\newcommand{\Csev}{\mathbb{C}_7}

\title{$\G$-invariant $G_2$-instantons on the \\ AC limit of the $\mathbb{C}_7$ family}
\author{Karsten Matthies, Johannes Nordstr\"om, Matt Turner}
\date{ }

\begin{document}
\maketitle
\newtheorem{theorem}[equation]{Theorem}
\newtheorem{lemma}[equation]{Lemma}
\newtheorem{proposition}[equation]{Proposition}
\newtheorem*{maintheorem*}{Main Theorem}

\theoremstyle{definition}
\newtheorem{ex}[equation]{Example}
\newtheorem{remark}[equation]{Remark}
\numberwithin{equation}{section}

\begin{abstract}
We construct $\G$-invariant $G_2$-instantons on the asymptotically conical limit of the $\Csev$ family of $G_2$-metrics. The construction uses a dynamical systems approach involving perturbations of an abelian solution and a solution on the $G_2$-cone. From this we obtain a 1-parameter family of invariant instantons with gauge group $\SU(2)$ and bounded curvature. 
\end{abstract}

\frenchspacing
\setlength{\parskip}{.4cm}

\section{Introduction}

Let $(M,\varphi)$ be a $G_2$-manifold, that is, a Riemannian 7-manifold with a closed and coclosed 3-form $\varphi$ that induces a metric with holonomy group the exceptional Lie group $G_2$. Now, let $P\to M$ be a principal $G$-bundle, where $G$ is assumed to be a compact semisimple Lie group. A connection $A$ on $P$ is a $G_2$-instanton if \begin{align}\label{eq:insteq}F_A\wedge {*}\varphi=0\end{align} or equivalently, if \begin{align}\label{eq:ASDanalogue}F_A\wedge\varphi=-{*}F_A\end{align}
where $F_A$ is the curvature of the connection $A$ and $*$ denotes the hodge star operator with respect to the metric on $M$ induced by $\varphi$.

We note that \eqref{eq:ASDanalogue} is analogous to the notion of an anti-self-dual (ASD) connection $A$ on a 4-manifold, defined as satisfying $F_A=-{*}F_A$. The moduli space of such connections can be used to define numerical invariants for smooth 4-manifolds; this is known as Donaldson Theory. The analogy of $G_2$-instantons with ASD connections motivates study for using $G_2$-instantons to construct enumerative invariants of $G_2$-manifolds, an idea expressed by Donaldson and Thomas in \cite{DT98}, and developed further by Donaldson and Segal in \cite{DS11}. However, the moduli space of $G_2$-instantons on a $G_2$-manifold is less well-behaved than the ASD moduli space under deformations of the metric. Therefore, even after finding a compactification of the 0-dimensional moduli space, simply counting the points in this compactification is not expected to be an invariant of the $G_2$-structure. 

On the other hand, it is conjectured in \cite{DS11} that adding a counterterm, which would be defined by counting, with weights, $G_2$-instantons on a different bundle together with associative submanifolds and some sort of multiplicity, would provide an invariant under deformation of the $G_2$-structure. This is a very difficult problem, and hence it is natural to search for examples of $G_2$-instantons on all known constructions of $G_2$-manifolds. 

$G_2$-instantons have attracted interest from both the mathematics and theoretical physics communities. A physical motivation for understanding moduli spaces of $G_2$-manifolds is their application to M-theory, a branch of String Theory which unifies the various versions of Superstring Theory. In one model, the universe is postulated to have eleven dimensions, four of which consist of Minkowski space-time. The remaining seven take the form of a compact $G_2$-manifold with diameter of order the Planck length. 

The goal of this paper is to provide examples of $G_2$-instantons on an asymptotically conical (AC) manifold $M_{1,1}$, i.e. a manifold which is diffeomorphic to a $G_2$-cone outside of a compact set, and whose metric is in some sense asymptotic to a conical metric. This manifold forms the AC limit of the so-called $\Csev$ family of complete asymptotically locally conical (ALC) $G_2$-metrics. This family is just one of infinitely many constructed by Foscolo, Haskins and Nordstr\"om in \cite{FHN18}; we will refer to the collection of these families as the infinite extension of $\mathbb{C}_7$. Each family lies on a manifold $M_{m,n}$, for coprime positive integers $m$ and $n$. $M_{m,n}$ is the circle bundle over the canonical bundle on $\CP^1\times\CP^1$, whose restriction to the zero section has first Chern class $c_1=(m,-n)$ under the isomorphism between the second cohomology and $\Z\times\Z$.

The $\mathbb{C}_7$ family and its infinite extension also admit a cohomogeneity one action of the group $\G$. A cohomogeneity one manifold is a Riemannian manifold with an isometric Lie group action whose generic orbits, known as the principal orbits, have codimension one. With this high level of symmetry, we consider $\G$-invariant $G_2$-instantons on the AC manifolds $M_{m,n}$; this reduces the system of PDEs given by \eqref{eq:insteq} to a system of nonlinear ODEs, which are generally simpler to solve or understand qualitatively. When $m=n=1$, the equations simplify further, and so we focus here on that particular case.

The main result of this paper is the existence of a 1-parameter family of $G_2$-instantons with gauge group $\SU(2)$ on the AC manifold $M_{1,1}$. While these solutions are not explicit, the construction gives some details about their qualitative behaviour; indeed, the solutions have bounded curvature and are asymptotic to dilation-invariant solutions on the $G_2$-cone.

\subsubsection*{Plan of the Paper}

We begin in Section \ref{sec:prel} by setting out some of the preliminary material required for the rest of the paper. Initially, we describe what it means for a manifold to be AC or ALC; all manifolds considered in this paper will fit into one of these two categories. In addition, all the manifolds will have a cohomogeneity one group action. We describe such manifolds and give the construction of homogeneous bundles on their orbits. Finally, we rewrite the $G_2$-instanton equations as evolution equations on a cohomogeneity one manifold.

We describe the manifolds of the $\Csev$ family and its infinite extension in Section \ref{sec:C7}. Firstly, we set out the underlying topological manifold of each family, parameterised by coprime positive integers $m$ and $n$. This includes detailing their cohomogeneity one structure; the principal orbits are $(S^3\times S^3)/\Z_{2(m+n)}$ and the singular orbit is diffeomorphic to $S^2\times S^3$. We then recount the construction from \cite{FHN18} of a complete torsion-free $G_2$-structure on each $M_{m,n}$. 

In Section \ref{sec:inst}, we start by applying the general $G_2$-instanton evolution equation from Section \ref{sec:prel} to the specific case of $M_{m,n}$. This yields a general system of ODEs whose solutions are $\SU(2)^2$-invariant $G_2$-instantons. Before specialising this system to a specific connection on a given homogeneous bundle, we classify all possible $\SU(2)$-bundles on $M_{m,n}$ which admit $\G$-invariant connections. This is a standard process which applies Wang's Theorem (see Theorem \ref{thm:Wang}). Once we know the possible bundles, we can give an explicit form of the $G_2$-instanton equations as a system of ODEs on the space of principal orbits. 

The next step is to determine which solutions of this system extend smoothly over the singular orbit; this is outlined in Section \ref{sec:init}. This is also a standard technique in the study of cohomogeneity one geometry. After finding a family of local solutions close to the singular orbit, we apply Eschenburg-Wang's results in Appendix A to determine which of these solutions extend smoothly over the singular orbit. The remainder of the paper aims to construct global solutions which extend such local solutions to the whole manifold.

The main result is given in Section \ref{sec:sols}. There are some elementary solutions which can be described explicitly in terms of the functions defining the AC metric; we describe them in Section \ref{sec:triv}. These solutions are either invariant flat connections or are abelian instantons. The former connections are flat but are not equivalent to the trivial connection via invariant gauge transformations, i.e. gauge transformations which are constant on each orbit of the cohomogeneity one action. The latter are solutions which arise from a connection on a U(1)-subbundle, where the Lie algebra structure is trivial, i.e. all of the Lie brackets vanish. Hence, the equations simplify considerably and we find a 1-parameter family of invariant abelian instantons. Only one member of this family has bounded curvature and we call this special abelian instanton $\Aab$.

The main result of the paper is the following theorem; see Theorem \ref{thm:main} in Section \ref{sec:dyn} for a more precise statement.

\begin{maintheorem*}
Consider the AC $G_2$-manifold $M_{1,1}$ and let $\Aab$ be the invariant abelian $G_2$-instanton with bounded curvature. There is a 1-parameter family of $\G$-invariant $G_2$-instantons with full gauge group $\SU(2)$ and bounded curvature, which is a small perturbation of $\Aab$ near the singular orbit $S^2\times S^3$. These instantons are asymptotic to a dilation-invariant solution on the cone. 
\end{maintheorem*}

This result is proved via a dynamical systems approach. We consider a solution which takes the form of a heteroclinic orbit, i.e. a solution whose path in phase space joins two distinct fixed points. One of these fixed points corresponds to the connection on an $\SU(2)$-bundle over each principal orbit known as the canonical invariant connection. The abelian solution $\Aab$ which we construct is asymptotic to this connection. The argument proceeds by perturbing both solutions and proving we can flow from one to the other. The resulting solution is smooth and has initial data that forms a subset of the local solutions which extend smoothly over the singular orbit. On the other hand, the limit at the conical end is the same fixed point as the end of the heteroclinic orbit. The AC manifolds $M_{m,n}$ are asymptotic to the $G_2$-cone over a finite quotient of the nearly K\"ahler manifold $S^3\times S^3$. This fixed point corresponds to the so-called canonical Hermitian connection on this finite quotient.

\subsubsection*{Acknowledgements}

Special thanks to Lorenzo Foscolo, Jason Lotay and Jakob Stein for their helpful comments and discussions. Thanks also to the reviewers of this paper for their constructive comments. This work was funded by the EPSRC Studentship 2106787 and the Simons Collaboration on Special Holonomy in Geometry, Analysis and Physics (grant \#488631, Johannes Nordstr\"om).

\section{Preliminaries}\label{sec:prel}

\subsection{AC and ALC \texorpdfstring{$G_2$}{G2}-manifolds}

The manifolds $M_{m,n}$ of interest in this paper each admit a holonomy $G_2$ metric with asymptotically conical geometry. A non-compact $G_2$-manifold $(M,\varphi)$ is asymptotically conical (AC) if $M$ is complete with one end where the $G_2$-structure $\varphi$ is asymptotic to a conical $G_2$-structure $\varphi_C$ on $\R^+_t\times N^6$, for some nearly K\"ahler manifold $(N,\omega,\Omega)$. More precisely, $N$ has a non-degenerate 2-form $\omega$ and a complex volume form $\Omega$ satisfying 
\[d\omega=3\Reom, \hspace{1cm} d\Imom=-2\omega^2\]
and the conical $G_2$-structure is given by
\[\varphi_C=dt\wedge\omega+\Reom.\]
Then there is a compact subset $Y\subset M$, an $R>1$ and a diffeomorphism $f:(R,\infty)\times N\to M\backslash Y$ such that $f^*(\varphi)$ approaches $\varphi_C$ with rate $\nu<0$:
\[\norm{\nabla^j(f^*(\varphi)-\varphi_C)}=O(t^{\nu-j}), \hspace{1cm} \forall\,j\geq 0\]
on $(R,\infty)\times N$. $M$ has only one end due to the Cheeger-Gromoll Splitting Theorem.

In addition, the manifolds $M_{m,n}$ have a family of asymptotically locally conical (ALC) metrics; for context, we briefly describe metrics which have ALC geometry. Such manifolds, under certain decay conditions, are asymptotic outside a compact set to a model metric on a circle bundle over a 6-dimensional cone; we construct this model metric as follows. Let $\pi:\Sigma^6\to X^5$ be a circle bundle, $h$ a metric on $X$, $\theta$ a connection on $\Sigma$ and $\ell$ a positive real number. Then we can construct a U(1)-invariant $G_2$-structure $\varphi_{\infty}$ on the cone $(1,\infty)\times\Sigma$ with associated metric
\[g_{\varphi_{\infty}}=dt^2+\ell^2\theta^2+t^2\pi^*h.\]
This metric is a Riemannian submersion over the cone $(1,\infty)\times X$, whose circle fibres have constant length $2\pi \ell$. As we take $\ell\to\infty$, the asymptotic geometry transitions and we recover the AC metric described above. 

\subsection{Homogeneous Bundles on Cohomogeneity One Manifolds}
\label{sec:hombund}

Now let $M$ be any Riemannian manifold and $G$ be a Lie group. $M$ is a cohomogeneity one manifold if $G$ acts isometrically on $M$ with generic orbits of codimension one. In the following, we consider such manifolds $M$ which are complete irreducible and Ricci flat, with $G$ compact; then $M$ is non-compact with one end and so $M/G$ is a half line $[0,\infty)$. We say that the orbits over points in $(0,\infty)$ are principal orbits, while the orbit corresponding to 0 is called the singular orbit. 

We can encode the cohomogeneity one structure of $M$ in its group diagram: we write $K_0\subset K\subset G$, where $K_0$ and $K$ are the stabilisers of the $G$-action on $M$ restricted to the principal orbits and singular orbit respectively. There is a $K$-representation $V$ of dimension $\dim(K)-\dim(K_0)+1$ such that $K$ acts transitively on the unit sphere with stabiliser $K_0$; then $M=G\times_KV$.

A principal $H$-bundle $P$ on a homogeneous manifold $G/K$ is called $G$-homogeneous if the action of $G$ on $G/K$ lifts to a $G$-action on $P$ which commutes with the action of $H$. Such bundles are determined by their isotropy homomorphism, which we now construct, as in \cite[p105]{KN63}. Let $u_0$ be an arbitrary point of $P$ over $x_0\in G/K$. The group $K$ is exactly the isotropy subgroup of the translation action of $G$ on $G/K$. Let $k\in K$; then $ku_0$ is a point in $P$, which lies in the same fibre as $u_0$. Thus, we can write $ku_0=u_0h$ for some $h\in H$. We define the isotropy homomorphism $\lambda:K\to H$ by $\lambda(k) = h$; it is proved in Section II.11 of \cite{KN63} that this is indeed a homomorphism.

Conversely, given a homomorphism $\lambda:K\to H$, the associated $H$-bundle 
\begin{equation}
\label{eq:lambundle}
    P_{\lambda}=G\times_{(K,\lambda)}H
\end{equation}
is a $G$-homogeneous $H$-bundle on $G/K$ whose isotropy representation is $\lambda$.
The Lie algebra $\mathfrak{g}$ of $G$ has an $\Ad_K$-invariant splitting $\mathfrak{g}=\mathfrak{k}\oplus\mathfrak{m}$. Then the canonical invariant connection on the bundle $G\to G/K$ is the invariant connection whose horizontal space at the identity in $G$ is $\mathfrak{m}$. It induces a corresponding canonical invariant connection on any $P_{\lambda}$, as in \cite{T09}, which is determined by the left-invariant translation of $d\lambda\oplus0:\mathfrak{k}\oplus\mathfrak{m}\to\mathfrak{h}$. We now refer to the following theorem from \cite{W58}, for a method of parameterising invariant connections on $P_{\lambda}$.

\begin{theorem}[Theorem 1, \cite{W58}]
\label{thm:Wang}
There is a 1-1 correspondence between $G$-invariant connections on $P_{\lambda}$ and morphisms of $K$-representations \begin{align}\label{eq:Lambda}\Lambda:(\mathfrak{m},\Ad)\to(\mathfrak{h},\Ad\circ\lambda).\end{align}
\end{theorem}

\begin{remark}
A $G$-invariant connection on $P_{\lambda}$ can be written as a 1-form $\omega\in\Omega^1(P_{\lambda},\mathfrak{h})$ (\cite[Section 2.3]{T09}), namely
\[\omega_{[g,h]}=\Ad_{h^{-1}}\circ\, (d\lambda+\Lambda)\circ(\Theta_G)_g+(\Theta_H')_h\]
where $\Theta_G$ is the left-invariant Maurer-Cartan form on $G$ and $\Theta_H'$ is the right-invariant Maurer-Cartan form on $H$. Then the difference between any invariant connection and the canonical invariant connection is determined by the morphism $\Lambda$, and the horizontal space of such a connection is determined by $\ker\Lambda$.
\end{remark}

Consider a cohomogeneity one manifold $M=G\times_KV$, with principal orbits $G/K_0$ and whose singular orbit is $G/K$.  Let $\lambda_s:K\to H$ be a group homomorphism and let {$\lambda_p:K_0\to H$} be the restriction of $\lambda$ to $K_0$.  Define \[P=G\times_K(V\times H),\]
a principal $H$-bundle over $M$ since the right action of $H$ on $P$ is free and $P/H=M$. The restriction to each principal orbit is $P_p:=P_{\lambda_p}$, the associated bundle to $\lambda_p$ given by (\ref{eq:lambundle}), and whose restriction to the singular orbit is $P_s:=P_{\lambda_s}$.

\subsection{\texorpdfstring{$G_2$}{G2}-instanton evolution equations}
\label{sec:insteveq}

We now consider the manifold $M=(0,\infty)\times N$, where $N$ is a 6-manifold admitting a 1-parameter family of half-flat $\SU(3)$-structures $(\omega,\Omega)$. An $\SU(3)$-structure $(\omega,\Omega)$ on a 6-manifold is a non-degenerate 2-form $\omega$ and a complex volume form $\Omega$ satisfying
\[\omega\wedge\Reom=0, \hspace{1cm} \frac{1}{6}\omega^3=\frac{1}{4}\Reom\wedge\Imom.\]
Such a structure is called half-flat if it also satisfies
\[d\omega\wedge\omega=d\Reom=0.\]
Then $M$ has a $G_2$-structure $\varphi=dt\wedge\omega+\Reom$ which is torsion-free if it satisfies the Hitchin flow, namely
\[\partial_t\Reom=d\omega, \hspace{1cm} \partial_t(\omega^2)=-2d\Imom.\]
The induced metric has the form $g=dt^2+g_t$ where $g_t$ is a 1-parameter family of metrics on $N$.

Let $P$ be a principal $G$-bundle on $M$, then $P$ is a pull-back of a bundle on $N$. We work in temporal gauge, so we assume that a connection on $P$ is of the form $A=\alpha(t)$, where $\alpha(t)$ is a 1-parameter family of connections on $P$. Then the curvature of $A$ is given by \[F_A=dt\wedge \dot{\alpha}+F_{\alpha}(t),\] where $F_{\alpha}(t)$ is the curvature of the connection $\alpha(t)$ on $P$ over $N$. In this setting, the $G_2$-instanton equation for the connection $A$ becomes an evolution equation for $\alpha(t)$, given by
\begin{align}\label{eqn:G2eqna}\dot{\alpha}\wedge\frac{1}{2}\omega^2-F_{\alpha}\wedge\Imom=0,\hspace{1cm} F_{\alpha}\wedge\frac{1}{2}\omega^2=0.\end{align}
The following lemma provides a useful reformulation of the above evolution equations.

\begin{lemma}[Lemma 1, \cite{LO18a}]
Let $M=(0,\infty)\times N$ be equipped with a $G_2$-structure $\varphi$ as above, satisfying $\omega\wedge d\omega=0$ and $\omega\wedge\partial_t\omega=-d\Imom$, or equivalently $d*\varphi=0$. Then $G_2$-instantons are in 1-1 correspondence with 1-parameter families of connections $\{\alpha(t)\}_{t\in I_t}$ solving the evolution equation
\begin{align}\label{eqn:G2eveq}
J_t\dot{\alpha}=-*_t(F_{\alpha}\wedge\Imom)\end{align}
subject to the constraint $\Lambda_t F_{\alpha}=0$, where $\Lambda_t$ denotes the metric dual of the operation of wedging with $\omega(t)$. Here, $*_t$ is the Hodge star operator associated to the $\SU(3)$-structure $(\omega(t),\Omega(t))$. Moreover, this constraint is compatible with the evolution: if it holds for some $t$, then it holds for all $t \in (0,\infty)$.
\end{lemma}

When deriving the $G_2$-instanton equations explicitly as an ODE system, it will be more convenient to use the formulation \eqref{eqn:G2eveq}. Such a derivation for the $G_2$-manifolds of interest in this paper is given in Section \ref{sec:eveqmn}.

\section{The \texorpdfstring{$\mathbb{C}_7$}{C7}-family and its infinite extension}\label{sec:C7}

We now give a comprehensive description of the $G_2$-structures on the manifolds $M_{m,n}$, constructed in \cite{FHN18}. The $\mathbb{C}_7$-family is given by the special case where $m=n=1$. The results of this paper focus on $M_{1,1}$ but we start by considering the general case as the setup is no harder to work with. We start by considering invariant half-flat structures on the principal orbits; we see that this structure, together with an extra U(1)-symmetry, give conditions on the $\SU(3)$-structure on each orbit. By exploiting the cohomogeneity one property of $M_{m,n}$, we see that the half-flat equations reduce to a system of ODEs. These simplifications were used in \cite{FHN18} to find a complete torsion-free AC $G_2$-structure on each $M_{m,n}$, for positive coprime $m$ and $n$.

\subsection{The manifolds \texorpdfstring{$M_{m,n}$}{Mmn}}\label{sec:Mmndef}

Fix a basis of left-invariant 1-forms $e_1,e_2,e_3, e_1',e_2',e_3'$ on $\SU(2)\times\SU(2)$ with the property that
\[de_i = -e_j\wedge e_k,\hspace{1cm} de_i'=-e_j'\wedge e_k'\]
for $(i,j,k)$ any cyclic permutation of $(1,2,3)$. Denote the dual vector fields by $E_i,E_i'$; they must satisfy $[E_i,E_j]=E_k$. Consider the maximal torus $T^2 \subset \SU(2)\times\SU(2)$; we may assume $T^2$ is generated by $E_3$ and $E_3'$.

For any $m,n\in\Z$, define subgroups $K_{m,n}$ of $T^2$ by
\[K_{m,n}=\{(e^{i\theta_1},e^{i\theta_2})\in T^2\ \vert\  e^{i(m\theta_1+n\theta_2)}=1\}.\]
We have $K_{m,n}\cong \textup{U}(1)\times \Z_{\gcd(m,n)}$ via \[\textup{U}(1)\times \Z_{\gcd(m,n)}\to K_{m,n}:(e^{i\theta},\zeta)\mapsto (e^{ik\theta}\zeta^r,e^{-ih\theta}\zeta^s),\] where $m=\gcd(m,n)h$, $n=\gcd(m,n)k$ and $rh+sk=1$. In particular, if $m,n$ are coprime, then $K_{m,n}\cong \textup{U}(1)$ via the map \[\textup{U}(1)\to K_{m,n}:e^{i\theta}\mapsto (e^{in\theta},e^{-im\theta}).\] In this case, we can fix integers $r,s\in\Z$ with $mr+ns=1$. Then, by the natural embedding of $K_{m,n}$ in $\SU(2)^2$, the left-invariant vector field given by $nE_3-mE_3'$ generates the $K_{m,n}$-action on $\SU(2)^2$.

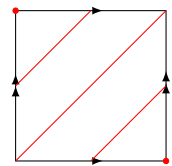
\begin{wrapfigure}{r}{0.2\textwidth}
\centering
\begin{tikzpicture}
\draw[postaction={decorate},decoration={
    markings,
    mark=at position .145 with {\arrow{latex}},
    mark=at position .380 with {\arrow{latex}},
    mark=at position .400 with {\arrow{latex}},
    mark=at position .605 with {\arrowreversed{latex}},
    mark=at position .855 with {\arrowreversed{latex}},
    mark=at position .875 with {\arrowreversed{latex}}
    }
  ]
    (0,-1) -- +(2,0) -- +(2,2) -- +(0,2) -- cycle;
    \draw[red] (0,-1) -- (2,1);
    \draw[red] (1,-1) -- (2,0);
    \draw[red] (0,0) -- (1,1);
    \filldraw[red] (0,1) circle (1pt);
    \filldraw[red] (2,-1) circle (1pt);
\end{tikzpicture}
\captionsetup{justification=centering}
\caption{\textcolor{red}{$K_{2,-2}$} \newline embedded in $T^2$}
\end{wrapfigure}

Take the canonical line bundle $B=K_{\CP^1\times\CP^1}$ over $\CP^1\times\CP^1$; this bundle is of interest since it is an example of an AC Calabi-Yau 3-fold which is asymptotic to a $\mathbb{Z}_2$-quotient of the conifold $C(\Sigma)$, where $\Sigma=\SU(2)^2/\Delta\text{U}(1)$ is endowed with its $\SU(2)\times\SU(2)$-invariant Sasaki-Einstein structure. There is a natural action of $\SU(2)$ on the line bundle $\mathcal{O}(-1)\to\CP^1$ which can be extended to an action of $\SU(2)^2$ on $K_{\CP^1\times\CP^1}$. Its group diagram is\[K_{2,-2}\subset T^2\subset\SU(2)\times\SU(2)\]
and the tangent cone at infinity is a free $\Z_2$-quotient of the conifold. Here, the choice of stabiliser $K_{2,-2}$ was made in \cite{FHN18} to be consistent with the physics literature.

For coprime positive integers $m$ and $n$, we now consider a non-trivial circle bundle $M_{m,n}$ over $B$, whose restriction to the zero section has first Chern class $c_1=m[\omega]-n[\omega']$, where $[\omega]$ and $[\omega']$ are generators of the second cohomology of the two $\CP^1$ factors. The natural action of each $\SU(2)$ factor on $\mathcal{O}(-1)$ gives a natural action of $\SU(2)^2$ on the circle bundle $M_{m,n}$. The total space is a simply connected cohomogeneity one 7-manifold, encoded in the group diagram
\[K_{m,n}\cap K_{2,-2}\subset K_{m,n}\subset \SU(2)\times\SU(2).\]

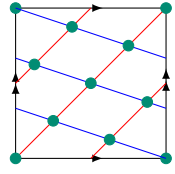
\begin{wrapfigure}{r}{0.2\textwidth}
\centering
\begin{tikzpicture}
\draw[postaction={decorate},decoration={
    markings,
    mark=at position .145 with {\arrow{latex}},
    mark=at position .380 with {\arrow{latex}},
    mark=at position .400 with {\arrow{latex}},
    mark=at position .605 with {\arrowreversed{latex}},
    mark=at position .855 with {\arrowreversed{latex}},
    mark=at position .875 with {\arrowreversed{latex}}
    }
  ]
    (0,-1) -- +(2,0) -- +(2,2) -- +(0,2) -- cycle;
    \draw[red] (0,-1) -- (2,1);
    \draw[red] (1,-1) -- (2,0);
    \draw[red] (0,0) -- (1,1);
    \filldraw[green!55!blue] (0,1) circle (2pt);
    \filldraw[green!55!blue] (2,-1) circle (2pt);
    
    \draw[blue] (0,-1/3) -- (2,-1);
    \draw[blue] (0,1/3) -- (2,-1/3);
    \draw[blue] (0,1) -- (2,1/3);
    \filldraw[green!55!blue] (2,1) circle (2pt);
    \filldraw[green!55!blue] (0,-1) circle (2pt);
    
    \filldraw[green!55!blue] (1,0) circle (2pt);
    \filldraw[green!55!blue] (1/2,-1/2) circle (2pt);
    \filldraw[green!55!blue] (3/2,1/2) circle (2pt);
    \filldraw[green!55!blue] (1/4,1/4) circle (2pt);
    \filldraw[green!55!blue] (7/4,-1/4) circle (2pt);
    \filldraw[green!55!blue] (3/4,3/4) circle (2pt);
    \filldraw[green!55!blue] (5/4,-3/4) circle (2pt);
\end{tikzpicture}
\captionsetup{justification=centering}
\caption{\textcolor{green!55!blue}{$K_{1,3}\cap K_{2,-2}$} \newline embedded in $T^2$}
\end{wrapfigure}

Note that $K_{m,n}\cap K_{2,-2}$ is isomorphic to $\Z_{2\vert m+n\vert}$, where this cyclic group is embedded in the maximal torus $T^2\subset\SU(2)\times\SU(2)$ via $\zeta\mapsto (\zeta^n,\zeta^{-m})$, for a primitive $2\vert{m+n}\vert$th root $\zeta$. See Figure 2 for an example, when $m=1$ and $n=3$: the four corner points are identified, leaving $2\vert m+n \vert=8$ distinct points.

Furthermore, the metrics constructed in Section \ref{sec:ab} will have a cohomogeneity one action of $\G$, where the additional U(1) acts on the circle fibres of $M_{m,n}$ by multiplication. It is easier to describe the manifolds $M_{m,n}$ via an $\SU(2)\times\SU(2)$ cohomogeneity one group action, but the requirement of the additional U(1)-invariance simplifies the equations for an invariant torsion-free $G_2$-structure.
\newline

\subsection{Construction of a complete AC torsion-free \texorpdfstring{$G_2$}{G2}-structure on \texorpdfstring{$M_{m,n}$}{Mmn}}
\label{sec:ab}

The main result of \cite{FHN18} that is of interest in our case is the existence of complete torsion-free AC $G_2$-metrics on the manifolds $M_{m,n}$. We start by considering $\G$-invariant solutions which exist in a neighbourhood of the singular orbit $Q$ of $M_{m,n}$, where $Q=(SU(2)\times\SU(2))/K_{m,n}\cong S^2\times S^3$. In the following set of propositions, we set out the argument of Foscolo, Haskins and Nordstr\"om that proves Theorem \ref{thm:7.1}. The theorem states that for each $m,n\in\Z$ and for a fixed $r_0>0$, there exists a $\beta>0$ such that the $G_2$-structure defined by
\begin{equation}\label{eq:G2beta}\varphi_{\beta}=-m^2r_0^3e_1\wedge e_2\wedge e_3 +n^2r_0^3e_1'\wedge e_2'\wedge e_3'+d(a(e_1\wedge e_1'+e_2\wedge e_2')+be_3\wedge e_3'),\end{equation}
for functions $a,b:[0,\infty)_t\to\R$ depending on $\beta$, is a complete torsion-free AC $G_2$-structure on $M_{m,n}$. Here, $t$ is the arc-length parameter along a geodesic meeting all principal orbits orthogonally.

The $G_2$-structure $\varphi_{\beta}$ on $M_{m,n}$ induces a metric $g=dt^2+g_t$. From \cite{MS12}, we have the following explicit formula for $g_t$:
\begin{align}
\label{eqn:metric}
g_t & =\frac{a(b+m^2r_0^3)}{\dot{a}\dot{b}}(e_1\otimes e_1 + e_2\otimes e_2) + \frac{a^2+bm^2r_0^3}{\dot{a}^2}e_3\otimes e_3 \\
 & \hspace{1cm}+ \frac{a(b+n^2r_0^3)}{\dot{a}\dot{b}}(e_1'\otimes e_1' + e_2'\otimes e_2') + \frac{a^2+bn^2r_0^3}{\dot{a}^2}e_3'\otimes e_3'\nonumber\\
&\hspace{1cm} - \frac{b^2-m^2n^2r_0^6}{2\dot{a}\dot{b}}(e_1\otimes e_1' + e_1'\otimes e_1 + e_2\otimes e_2' + e_2'\otimes e_2) + \frac{b^2-2a^2+m^2n^2r_0^6}{2\dot{a}^2}(e_3\otimes e_3' + e_3'\otimes e_3).\nonumber
\end{align}

\begin{proposition}[{\cite[{Proposition 4.5 (iii) and (iv)}]{FHN18}}]
\label{prop:4.5}
Fix coprime positive $m,n\in\Z$. $\G$-invariant $G_2$-structures on $M_{m,n}$ have the form $\varphi_{\beta}$ given in \eqref{eq:G2beta} for some functions $a,b:[0,\infty)_t\to\R$.
There exists a 2-parameter family of torsion-free $G_2$-structures $\varphi$ of this form defined in a neighbourhood of the singular orbit. The family is parameterised by $r_0\in\R$ and $\beta>0$ with
\begin{align*}
    a & = r_0^2\beta t + \frac{(m+n)\beta^3-m^{\frac{5}{2}}n^{\frac{5}{2}}}{12\sqrt{mn}\beta^3}t^3 + O(t^5),\\
    b & = mnr_0^3 +\frac{\sqrt{mn}(m+n)r_0}{2\beta}t^2+\frac{m+n}{96r_0\beta^5}\left(7-4(m+n)\beta^3\right)t^4+O(t^6).
\end{align*}
\end{proposition}

Having described local solutions in a neighbourhood of the singular orbit, we follow \cite{FHN18} and consider asymptotic local solutions at infinity. Their results are given in the following proposition.

\begin{proposition}[{\cite[Proposition 5.3 (ii)]{FHN18}}]
\label{prop:5.3}
Let $C$ be the $G_2$-holonomy cone over the homogeneous nearly K\"ahler structure on $S^3\times S^3$ and set $\nu_{\infty}=\frac{\sqrt{145}+7}{2}$. Then for every $c\in \R$, there exists a unique $\G$-invariant torsion-free $G_2$-structure 
$\varphi_{\beta}$ given by \eqref{eq:G2beta} on $(T,\infty)\times S^3\times S^3$ for some $T>0$, where the functions $t^{-3}a$ and $t^{-3}b$ admit convergent generalised power series expansions in powers of $t^{-3}$ and $t^{-\nu_{\infty}}$ satisfying
\[\frac{54}{\sqrt{3}}t^{-3}a=1+O(t^{-3}),\hspace{1cm} \frac{54}{\sqrt{3}}t^{-3}b=1+O(t^{-3}),\hspace{1cm} \frac{54}{\sqrt{3}}t^{-3}(b-a)=ct^{-\nu_{\infty}}+O(t^{-12}).\]
In particular, the associated metric $g_{\varphi}$ has a complete asymptotically conical end as $t\to\infty$ asymptotic to the cone $C$ with rate $-3$.
\end{proposition}

We are now ready to state the main relevant theorem of \cite{FHN18}, which gives existence of a complete torsion-free AC $G_2$-structure on $M_{m,n}$.

\begin{theorem}[{\cite[Theorem 7.1 (ii)]{FHN18}}]
\label{thm:7.1}
Fix coprime positive $m,n\in\Z$, and a real number $r_0>0$. For $\beta>0$, let $\varphi_{\beta}$ be the (locally defined) $\G$-invariant torsion-free $G_2$-structure of Proposition \ref{prop:4.5} closing smoothly on $\SU(2)^2/K_{m,n}$ and satisfying
\[a=r_0^3\beta t + O(t^3), \hspace{1cm} b=mnr_0^3+O(t^2)\]
as $t\to 0$. There exists $\beta_{ac}>0$ such that $\varphi_{\beta_{ac}}$ extends to a complete torsion-free AC $G_2$-structure asymptotic to the cone over the $\Z_{2(m+n)}$-quotient of the homogeneous nearly K\"ahler structure on $S^3\times S^3$ with rate $-3$.
\end{theorem}

\begin{remark}
\label{rem:abcond}
In Section 7.1 of \cite{FHN18}, the solutions $a$ and $b$ for the AC $G_2$-structure of Theorem \ref{thm:7.1} are proved to satisfy
\begin{gather*} b>a>0, \quad \dot{a}>\dot{b}>0,\quad b>\max(-m^2r_0^3,-n^2r_0^3), \quad b>mnr_0^3, \\  4a^2(b+m^2r_0^3)(b+n^2r_0^3)>(b^2-m^2n^2r_0^6)^2, \quad ka>\frac{b^2-m^2n^2r_0^6}{\sqrt{(b+m^2r_0^3)(b+n^2r_0^3)}} \text{ for } k\in(1,2)\end{gather*}
for $t\in(0,\infty)$. This is currently the only information we have about $a$ and $b$ globally.
\end{remark}

These solutions provide part of the classification of complete $\G$-invariant simply-connected $G_2$-manifolds. The full classification is given in the following theorem.

\begin{theorem}[{\cite[Theorem 7.3]{FHN18}}]\label{thm:class}
Let $(M,g)$ be a complete $\G$-invariant $G_2$-metric with $M$ simply-connected. Then $(M,g)$ is isometric to one of the following complete metrics:
\begin{enumerate}
    \item[\textup{(i)}] the complete metrics of Theorem \ref{thm:7.1};
    \item[\textup{(ii)}] Byrant-Salamon's complete $\SU(2)^3$-invariant AC $G_2$-metric on $S^3\times\R^4$ - see \cite{BS89};
    \item[\textup{(iii)}] the $\mathbb{B}_7$ family of complete $\G$-invariant ALC $G_2$-metrics on $S^3\times\R^4$ - see \cite{B13} and \cite{BGGG};
    \item[\textup{(iv)}] the $\mathbb{D}_7$ family\footnote{This family was conjectured in the physics literature - see \cite{B02} and \cite{CGLP}.} of complete $\G$-invariant ALC $G_2$-metrics on $S^3\times\R^4$ - see \cite[Theorem 6.17 (i)]{FHN18}; 
    \item[\textup{(v)}] if $\beta>\beta_{ac}$ in Theorem \ref{thm:7.1}, then $\varphi_{\beta}$ extends to a complete torsion-free ALC $G_2$-structure asymptotic to a circle bundle over a $\Z_2$-quotient conifold - see \cite[Theorem 7.1 (i)]{FHN18}.
\end{enumerate}
\end{theorem}

Some progress has been made to construct $G_2$-instantons on these manifolds: Clarke \cite{C13}, Lotay and Oliveira \cite{LO18b} describe some solutions of the $G_2$-instanton equations for the Bryant Salamon metric (ii), and some progress is made towards describing the solutions for the family (iii). The existence of solutions in cases (iv) and (v) is being considered by the third author. We now consider the problem of constructing $G_2$-instantons on (i), the complete metrics of Theorem \ref{thm:7.1}.

\section{\texorpdfstring{$\G$}{SU(2)xSU(2)xU(1)}-invariant \texorpdfstring{$G_2$}{G2}-instantons equations}\label{sec:inst}

Having described the manifolds $M_{m,n}$, we now turn our attention to the construction of $G_2$-instantons. We start by writing the instanton equations for a manifold admitting a  cohomogeneity one action of $\SU(2)^2$, before specialising further to $M_{m,n}$. We next find a general form for the $\SU(2)^2$-invariant connections on the principal orbits before specialising to the case of $\G$-invariant instantons.

\subsection{\texorpdfstring{$G_2$}{G2}-instanton evolution equations for \texorpdfstring{$M_{m,n}$}{Mmn}}\label{sec:eveqmn}

We saw in Section \ref{sec:insteveq} that any $G_2$-instanton on the manifold $M = (0,\infty)\times N$, defined as a 1-parameter family of connections $\alpha(t)$ on the principal orbits $N$, must satisfy the equation
\[J_t\dot{\alpha}=-*_t(F_{\alpha}\wedge\Imom),\]
subject to the constraint $\Lambda_t F_{\alpha}=0$. We now derive these equations in the case where $N=\SU(2)^2/\Z_{2(m+n)}$.

We have a formula for $\Imom$ from \cite{FHN18}, namely
\begin{align*}
    2\dot{a}^2\dot{b}\Imom\ \ =&\ \ (2a^2b + m^2r_0^3(2a^2+b^2-m^2n^2r_0^6))\,e_{123} + (2a^2b + n^2r_0^3(2a^2+b^2-m^2n^2r_0^6))\,e_{1'2'3'}\\
    & - (a(b^2+m^2n^2r_0^6)+2abn^2r_0^3)(e_{12'3'}+e_{23'1'})\\
    & + (b(b^2-2a^2 - m^2n^2r_0^6)-2a^2n^2r_0^3)\,e_{31'2'} + (b(b^2-2a^2 - m^2n^2r_0^6)-2a^2m^2r_0^3)\,e_{3'12}\\
    & - (a(b^2+m^2n^2r_0^6)+2abm^2r_0^3)(e_{1'23}+e_{2'31})
\end{align*}
where $e_{123}$ denotes $e_1\wedge e_2\wedge e_3$, $e_{1'2'3'}$ denotes $e_1'\wedge e_2'\wedge e_3'$ and so on. The coefficients on the right hand side of this equation appear frequently in the derivation and depend only on the $G_2$-structure on $M_{m,n}$, so we choose to rename them as follows:
\begin{align*}
   \Phi_m=2a^2b + m^2r_0^3(2a^2+b^2-m^2n^2r_0^6),\hspace{2cm} & \Phi_n = 2a^2b + n^2r_0^3(2a^2+b^2-m^2n^2r_0^6),\\
    \Psi_m = b(b^2-2a^2 - m^2n^2r_0^6)-2a^2m^2r_0^3,\hspace{2cm} & \Psi_n = b(b^2-2a^2 - m^2n^2r_0^6)-2a^2n^2r_0^3,\\
    \chi_m = a(b^2+m^2n^2r_0^6)+2abm^2r_0^3,\hspace{2cm} & \chi_n = a(b^2+m^2n^2r_0^6)+2abn^2r_0^3.
\end{align*}
So we may write
\[2\dot{a}^2\dot{b}\Imom = \Phi_m\,e_{123} + \Phi_n\,e_{1'2'3'} + \Psi_m\,e_{3'12} + \Psi_n\,e_{31'2'} - \chi_m(e_{1'23}+e_{2'31}) - \chi_n(e_{12'3'}+e_{23'1'}).\]

We now turn our attention to $F_{\alpha}$. As described in Section \ref{sec:hombund}, any isotropy homomorphism $\lambda_p^j:\Z_{2(m+n)}\to G$ is given by $\lambda_p^j(\zeta)=\text{diag}(\zeta^j,\zeta^{-j})$ for $j\in\{0,1,\dots,2(m+n)-1\}$ and has an associated $\SU(2)^2$-homogeneous $G$-bundle
\[P_p^j=\SU(2)^2\times_{(\Z_{2(m+n)},\lambda_p^j)}G\]
over $\SU(2)^2/\Z_{2(m+n)}$. A connection on $P_p^j$ can be pulled back to the trivial bundle $\SU(2)^2\times G$, so the set of connections on $P_p^j$ can be viewed as a subset of those on the trivial bundle. Hence, we can write such a connection as a 1-form on $\SU(2)^2$ with values in $\mathfrak{g}$, satisfying the condition given in Theorem \ref{thm:Wang}. The canonical invariant connection $\alpha^{\text{can}}$ has horizontal space $\mathfrak{m}=\su(2)\oplus\su(2)$, so its connection 1-form as an element of $\Omega^1(\SU(2)^2,\mathfrak{g})$ is zero. Then Theorem \ref{thm:Wang} tells us that any other $\SU(2)^2$-invariant connection differs from $\alpha^{\text{can}}$ by a morphism of $\Z_{2(m+n)}$-representations $\Lambda:(\su(2)\oplus\su(2),\Ad)\to(\mathfrak{g},\Ad\circ\lambda_p^j)$. We can extend such a $\Lambda$ by left-invariance to $\SU(2)^2$, which leads to a 1-form with values in $\mathfrak{g}$. Then on $M$, the most general $\SU(2)^2$-invariant connection on $P_p^j$ in temporal gauge can be written as 
\[\alpha(t) = \sum_{i=1}^3 \alpha_i(t)\otimes e_i + \sum_{i=1}^3\alpha_i'(t)\otimes e_i',\]
for functions $\alpha_i,\alpha_i':[0,\infty)\to\mathfrak{g}$ such that for each $t$, $\alpha_i(t)$ and $\alpha_i'(t)$ satisfy conditions corresponding to the $\Z_{2(m+n)}$-equivariance of $\Lambda$.

A quick calculation gives
\begin{align}\label{eqn:curvature}
F_{\alpha}= \sum_{i=1}^3[\alpha_i,\alpha_i']\,e_{ii'} + \sum_{i=1}^3 (([\alpha_j,\alpha_k]-\alpha_i)\,e_{jk}+([\alpha_j',\alpha_k']-\alpha_i')\,e_{j'k'}) + \sum_{i=1}^3([\alpha_j,\alpha_k']\,e_{jk'} + [\alpha_j',\alpha_k]\,e_{j'k})\end{align}
where $(i,j,k)$ is a cyclic permutation of (1,2,3).

To calculate the right hand side of (\ref{eqn:G2eveq}), we now compute
\begin{align*}
    2\dot{a}^2\dot{b}F_{\alpha}\wedge\Imom\ \ = &\ \  \Big(([\alpha_2,\alpha_3]-\alpha_1)\Phi_n + [\alpha_2',\alpha_3]\chi_n - [\alpha_2,\alpha_3']\Psi_n - ([\alpha_2',\alpha_3']-\alpha_1')\chi_m\Big)e_{231'2'3'}\\
    & -\Big(([\alpha_3,\alpha_1]-\alpha_2)\Phi_n + [\alpha_3,\alpha_1']\chi_n - [\alpha_3',\alpha_1]\Psi_n - ([\alpha_3',\alpha_1']-\alpha_2')\chi_m\Big)e_{131'2'3'}\\
    & + \Big(([\alpha_1,\alpha_2]-\alpha_3)\Phi_n + ([\alpha_1',\alpha_2] + [\alpha_1,\alpha_2'])\chi_n + ([\alpha_1',\alpha_2']-\alpha_3')\Psi_m\Big)e_{121'2'3'}\\
    & + \Big(([\alpha_2',\alpha_3']-\alpha_1')\Phi_m + [\alpha_2,\alpha_3']\chi_m - [\alpha_2',\alpha_3]\Psi_m - ([\alpha_2,\alpha_3]-\alpha_1)\chi_n\Big)e_{1232'3'}\\
    & - \Big(([\alpha_3',\alpha_1']-\alpha_2')\Phi_m + [\alpha_3',\alpha_1]\chi_m - [\alpha_3,\alpha_1']\Psi_m - ([\alpha_3,\alpha_1]-\alpha_2)\chi_n\Big)e_{1231'3'}\\
    & + \Big(([\alpha_1',\alpha_2']-\alpha_3')\Phi_m + ([\alpha_1,\alpha_2']+ [\alpha_1',\alpha_2])\chi_m + ([\alpha_1,\alpha_2]-\alpha_3)\Psi_n\Big)e_{1231'2'}.\\
\end{align*}
For the left hand side of (\ref{eqn:G2eveq}), we note that the complex structure $J_t$ is such that 
\[J_te_i^{(\prime)} = -*\left(e_i^{(\prime)}\wedge \frac{1}{2}\omega^2\right).\]
Thus, we may simply compare coefficients of $F_{\alpha}\wedge\Imom$ and $\dot{\alpha}\wedge\frac{1}{2}\omega^2$, while keeping track of signs. This leads to the following general ODEs for the $\alpha_i(t),\alpha_i'(t)$
\begin{align}
\label{eqns:genODEs}
    2\dot{a}^3\dot{b}^2\dot{\alpha}_1\ \ =\ \ &
    (\alpha_1'-[\alpha_2',\alpha_3'])\Phi_m - [\alpha_2,\alpha_3']\chi_m + [\alpha_2',\alpha_3]\Psi_m + ([\alpha_2,\alpha_3]-\alpha_1)\chi_n,\nonumber \\[0.8ex]
    2\dot{a}^3\dot{b}^2\dot{\alpha}_2\ \ =\ \ &
    (\alpha_2'-[\alpha_3',\alpha_1'])\Phi_m - [\alpha_3',\alpha_1]\chi_m + [\alpha_3,\alpha_1']\Psi_m + ([\alpha_3,\alpha_1]-\alpha_2)\chi_n,\nonumber \\[0.8ex]
    2\dot{a}^4\dot{b}\dot{\alpha}_3\ \ =\ \ & (\alpha_3'-[\alpha_1',\alpha_2'])\Phi_m - ([\alpha_1,\alpha_2']+ [\alpha_1',\alpha_2])\chi_m + (\alpha_3-[\alpha_1,\alpha_2])\Psi_n, \\[0.8ex]
    2\dot{a}^3\dot{b}^2\dot{\alpha}_1'\ \ =\ \ &
    (\alpha_1-[\alpha_2,\alpha_3])\Phi_n - [\alpha_2',\alpha_3]\chi_n + [\alpha_2,\alpha_3']\Psi_n + ([\alpha_2',\alpha_3']-\alpha_1')\chi_m,\nonumber \\[0.8ex]
    2\dot{a}^3\dot{b}^2\dot{\alpha}_2'\ \ =\ \ &
    (\alpha_2-[\alpha_3,\alpha_1])\Phi_n - [\alpha_3,\alpha_1']\chi_n + [\alpha_3',\alpha_1]\Psi_n + ([\alpha_3',\alpha_1']-\alpha_2')\chi_m,\nonumber \\[0.8ex]
    2\dot{a}^4\dot{b}\dot{\alpha}_3'\ \ =\ \ & 
    (\alpha_3-[\alpha_1,\alpha_2])\Phi_n - ([\alpha_1',\alpha_2] + [\alpha_1,\alpha_2'])\chi_n + (\alpha_3'-[\alpha_1',\alpha_2'])\Psi_m.\nonumber
\end{align}

Finally, we have the constraint $\Lambda_tF_{\alpha}=0$, or equivalently $F_{\alpha}\wedge\frac{1}{2}\omega^2=0$ which yields the equation
\begin{align}\label{eqn:constraint}\dot{a}\dot{b}([\alpha_1,\alpha_1']+[\alpha_2,\alpha_2'])+\dot{a}^2[\alpha_3,\alpha_3']=0.\end{align}

We now restrict to the case of $\G$-invariant connections. The stabiliser of the $\G$-action on $M_{m,n}$ away from the singular orbit of the $\SU(2)^2$-action is a subgroup $K_0\subset T^2\times\text{U}(1)\subset\G$, the image of 
\[\Z_2\times \text{U}(1)\to T^2\times\text{U}(1):(\xi,e^{i\theta})\hookrightarrow(\xi e^{i\theta},e^{i\theta},\xi^{-m}e^{-i(m+n)\theta})\]
where $\xi\in\Z_2\subset$ U(1).
Hence, the principal orbits can be written as 
\begin{equation}\label{eq:isotoU(1)}\SU(2)^2/\Z_{2(m+n)}\cong \G/K_0\end{equation}
where this isomorphism is given by $(g_1,g_2)\mapsto (g_1,g_2,1)$.
Isomorphism classes of homogeneous $\SU(2)$-bundles on these principal orbits are in correspondence with conjugacy classes of isotropy homomorphisms (see Section \ref{sec:hombund}). Therefore, we need to consider isotropy homomorphisms $K_0\to\SU(2)$ which, up to conjugacy, will be of the form
\[\lambda_p^{l,k}(\xi,e^{i\theta})=\left(\begin{array}{cc} \xi^le^{ik\theta} & 0\\ 0 & \xi^{-l}e^{-ik\theta} \end{array}\right)\]
for $l\in\{0,1\}$ and $k\in\Z$.

Following the proof of Proposition 8 of \cite{LO18b}, we take the complement of the isotropy algebra $\mathfrak{u}(1)$ to be $\mathfrak{m}=\su(2)\oplus\su(2)\oplus\ 0$. The canonical invariant connection on the bundle
\[P_p^{l,k}=(\G)\times_{(K_0,\lambda_p^{j,k})}\SU(2)\]
is given by $d\lambda_p^{l,k}=E_3\otimes kd\theta$, where $\theta$ is the periodic coordinate on U(1). Theorem \ref{thm:Wang} states that any other invariant connection on $P_p^{l,k}$ can be written as $d\lambda_p^{l,k}+\Lambda_{l,k}$, where $\Lambda_{l,k}$ is the left invariant extension to $\G$ of a morphism of $K_0$-representations \begin{equation}\label{eq:Lambdajk}\Lambda_{l,k}:(\mathfrak{m},\Ad)\to(\su(2),\Ad\circ\lambda_p^{l,k}).\end{equation}
The adjoint action of $\Z_2$ on the Lie algebra is trivial, so we decompose $\mathfrak{m}$ and $\su(2)$ into irreducibles of U(1)-representations. We have $\mathfrak{m}=(\R\oplus\C_2)\oplus(\R\oplus\C_2)\oplus 0$ and $\su(2)=\R\oplus\C_{2k}$, so we see that other irreducible invariant connections exist only when $k=1$ by Schur's Lemma. Therefore, there are only 2 distinct bundles, parametrised by $l\in\{0,1\}$, over the principal orbits which admit irreducible invariant connections. 

The isomorphism \eqref{eq:isotoU(1)} induces bundle isomorphisms between $P_p^j$ and $P_p^{l,k}$ when $j+l(m+n)+mk\in 2(m+n)\Z$. When $k=1$, this is exactly when $l=0$, $j=-m$ or when $l=1$, $j=n$. Indeed, $P_p^j$ is the bundle $P_p^{l,k}$ where we forget the additional U(1)-action. Appendix \ref{sec:invtens} describes the possible extensions of the bundles $P_p^j$ over the singular orbit for the $\SU(2)^2$-action. We denote the pull-back of the two bundles $P_p^{0,1}$ and $P_p^{1,1}$ to $\R^+\times(\G)/K_0$ by $P_p^{-m}$ and $P_p^n$ respectively in what follows.

The singular orbit is of the form $\SU(2)^2/K_{m,n}$, where the $K_{m,n}$-action is generated by $nE_3-mE_3'$. Therefore, we can consider a new basis for the maximal torus $T^2$, given by 
\[nE_3-mE_3',\hspace{1cm} rE_3+sE_3'.\]
Here, the $r,s\in\Z$ were chosen in Section \ref{sec:Mmndef} to satisfy $mr+ns=1$; we initially choose an integral basis to obtain the results of Appendix \ref{sec:invtens} on the smooth extension of 1-forms over the singular orbit. The dual left-invariant 1-forms are $se_3-re_3'$ and $me_3+ne_3'$ respectively.

\begin{remark}\label{rem:mn}
When $m$ and $n$ are distinct, the $G_2$-instanton equations are more complicated. Explicit expressions for the abelian solutions are more elusive because the equations only decouple when $m=n=1$. Therefore, we restrict our attention to the manifold $M_{1,1}$ and we will consider the general case in future work.
\end{remark}

Having obtained the results of Appendix \ref{sec:invtens}, we no longer need an integral basis of coframes on $M_{m,n}$ and we can take any $r,s\in\R$ such that $mr+ns=1$. When $m=n=1$, the most obvious choice is $r=s=\frac{1}{2}$ and we work with this basis for the remainder of the paper.

\subsection{\texorpdfstring{$\G$}{SU(2)xSU(2)xU(1)}-invariant ODEs}\label{sec:11ODEs}

Recall from the discussion in the previous section that the additional U(1)-symmetry of $M_{1,1}$ can be encoded by writing the principal orbits in the form 
\[\SU(2)^2/\Z_4\cong \G/K_0.\]
The following proposition details the possible connections on such principal orbits.

\begin{proposition} \label{prop:11ODEs}
Let $A$ be an $\SU(2)^2\times\emph{U}(1)$-invariant $G_2$-instanton with gauge group $\SU(2)$ on the bundles $P_p^{\pm1}$ over $\R^+\times (\SU(2)^2/\Z_4) \cong \R^+\times(\SU(2)^2\times\emph{U}(1))/K_0$. Then $A$ takes the form:
\begin{equation}\label{eq:A}A= f(E_1\otimes e_1+E_2\otimes e_2)+f'(E_1\otimes e_1'+E_2\otimes e_2')+g(E_3\otimes \tfrac{1}{2}(e_3-e_3'))+h(E_3\otimes(e_3+e_3'))\end{equation}
with $f,f',g,h:\R^+\to\R$ satisfying,
\begin{align}
\label{eq:G2ode11}
    2\dot{a}^3\dot{b}^2\dot{f}\ \ =\ \ &
    f'(1-h+\tfrac{1}{2}g)\Phi_1 + f(g-1)\chi_1 + f'(\tfrac{1}{2}g+h)\Psi_1,\nonumber \\[0.8ex]
    2\dot{a}^3\dot{b}^2\dot{f}'\ \ =\ \ &
    f(1-\tfrac{1}{2}g-h)\Phi_1 - f'(g+1)\chi_1 + f(h-\tfrac{1}{2}g)\Psi_1,\nonumber \\[0.8ex]
    2\dot{a}^4\dot{b}\dot{g}\ \ =\ \ & (f^2-(f')^2-g)(\Phi_1-\Psi_1), \\[0.8ex]
    2\dot{a}^4\dot{b}\dot{h}\ \ =\ \ & 
    (h-\tfrac{1}{2}(f')^2-\tfrac{1}{2}f^2)(\Phi_1+\Psi_1) - 2ff'\chi_1.\nonumber
\end{align}
\end{proposition}
\begin{proof}
Recall that on each principal orbit, we have the bundles $P_p^{l,1}$ which correspond to the isotropy homomorphism $\lambda_p^{l,1}$. Linear maps $\Lambda_{l,1}$ given by \eqref{eq:Lambdajk} are morphisms of $K_0$-representations if they are equivariant. Any invariant connection on $P_p^{l,1}$ differs from the canonical invariant connection $d\lambda_p^{l,1}$ by the left-invariant extension of such a morphism $\Lambda_{l,1}$. Such a connection on each $P_p^{l,1}$ is given, up to gauge transformation, by 
\[A=E_3\otimes d\theta + f(E_1\otimes e_1+E_2\otimes e_2)+f'(E_1\otimes e_1'+E_2\otimes e_2')+g(E_3\otimes \tfrac{1}{2}(e_3-e_3'))+h(E_3\otimes(e_3+e_3'))\]
where $f,f',g,h$ are constants. Recall that the bundles $P_p^{0,1}$, $P_p^{1,1}$ over $\G/K_0$ are isomorphic to $P_p^{-1}$, $P_p^1$ over $\SU(2)^2/\Z_{4}$ respectively when we forget the additional U(1)-action. Thus any $\G$-invariant connection on $P_p^{\pm1}$ over $\R^+\times\SU(2)^2/\Z_4$ in temporal gauge has the form
\[A=\gamma\Big(f(E_1\otimes e_1+E_2\otimes e_2)+f'(E_1\otimes e_1'+E_2\otimes e_2')+g(E_3\otimes \tfrac{1}{2}(e_3-e_3'))+h(E_3\otimes(e_3+e_3'))\Big)\gamma^{-1}\]
where $f,f',g,h$ are functions $\R^+\to\R$, $\gamma:\R^+\to\SU(2)$. The presence of the function $\gamma$ is due to the dependence on $t$ of the gauge transformations on each principal orbit. We now show that requiring $A$ to be a $G_2$-instanton forces the function $\gamma$ to be constant, and hence the choice of gauge transformation does not depend on $t$. Using the substitutions
\begin{gather*}
\alpha_1=\gamma fE_1\gamma^{-1},\ \ \alpha_2=\gamma fE_2\gamma^{-1},\ \ \alpha_3=\gamma(\tfrac{1}{2}g+h)E_3\gamma^{-1},\\  \alpha_1'=\gamma f'E_1\gamma^{-1},\ \ \alpha_2'=\gamma f'E_2\gamma^{-1},\ \ \alpha_3'=\gamma(h-\tfrac{1}{2}g)E_3\gamma^{-1},
\end{gather*}
the general ODEs \eqref{eqns:genODEs} transform to ODEs in $f,f',g,h$ and $\gamma$. For example, the ODE for $\alpha_1$ becomes
\[2\dot{a}^3\dot{b}^2\dot{f}E_1 + [\gamma^{-1}\dot{\gamma},E_1]f\ \ =\ \ 
    f'(1-h+\tfrac{1}{2}g)\Phi_1E_1 + f(g-1)\chi_1E_1 + f'(\tfrac{1}{2}g+h)\Psi_1E_1.\]
Hence $[\gamma^{-1}\dot{\gamma},E_1]f=0$ and since this type of term appears in each ODE and we assume $A\neq 0$, then we must have $\dot{\gamma}=0$, and hence we can write $A$ as 
\[A=f(E_1\otimes e_1+E_2\otimes e_2)+f'(E_1\otimes e_1'+E_2\otimes e_2')+g(E_3\otimes \tfrac{1}{2}(e_3-e_3'))+h(E_3\otimes(e_3+e_3')),\]
with $f,f',g,h$ satisfying the ODEs given in the statement of the proposition. The constraint \eqref{eqn:constraint} is trivially satisfied by this choice of gauge.
\end{proof}

The metric $g_t$ on the principal orbits of $M_{1,1}$ is invariant under the involution generated by the outer automorphism of $\SU(2)^2$ that swaps the two factors. The defining functions of the connection form $A$ transform under this involution by
\[(f,f',g, h)\mapsto (f',f,-g,h).\] 
We could alternatively embed the stabiliser $\Z_2\times\text{U}(1)$ in $\G$ via
\[(\xi,e^{i\theta})\hookrightarrow(e^{i\theta},\xi e^{i\theta},\xi^{-1}e^{-2i\theta}).\]
By choosing a different embedding for the bundles $P_p^1$ and $P_p^{-1}$, the automorphism induces an $\SU(2)$-equivariant bundle map $P_p^{-1}\to P_p^1$. Since the action of $\G$ is respected by this bundle map, the pull-back of an $\G$-invariant connection on $P_p^{-1}$ is itself an $\G$-invariant connection on $P_p^1$. Thus, we need only consider the bundle $P_p^1$ on $\R^+\times\SU(2)^2/\Z_4$.

\section{Initial Conditions}\label{sec:init}

We now consider the conditions for extending the connections of Section \ref{sec:11ODEs} over the singular orbit. Here, we apply the results of Appendix A to write local solutions around the singular orbit as power series expansions. Near $t=0$, the instanton equations \eqref{eq:G2ode11} take the form of a singular initial value problem, and so we state the following result which provides local solutions to such problems.

\begin{theorem}[{\cite[Theorem 4.3]{FHN18}}]
\label{thm:4.3}
Consider the singular initial value problem
\[\dot{y}=\frac{1}{t}M_{-1}(y)+M(t,y),\hspace{1cm} y(0)=y_0,\]
where $y$ takes values in $\R^k$, $M_{-1}:\R^k\to\R^k$ is a smooth function of $y$ in a neighbourhood of $y_0$ and $M:\R\times\R^k\to\R^k$ is smooth in $t,y$ in a neighbourhood of $(0,y_0)$. Assume that
\begin{enumerate}
    \item[\textup{(i)}] $M_{-1}(y_0)=0$;
    \item[\textup{(ii)}] $hId-d_{y_0}M_{-1}$ is invertible $\forall\,h\in\mathbb{N}$, $h\geq 1$.
\end{enumerate}
Then there exists a unique solution $y(t)$ in a sufficiently small neighbourhood of 0. Furthermore, $y$ depends continuously on $y_0$ satisfying $(i)$ and $(ii)$.
\end{theorem}

We now study the conditions for a connection $A$ of the form \eqref{eq:A} to extend smoothly over the singular orbit \[\SU(2)^2/K_{1,1}\cong(\G)/K\]
where $K\cong$ U(1)$^2$ is the stabiliser of the $\G$-action on the singular orbit and is embedded in $\G$ by 
\[(e^{i\psi},e^{i\theta})\hookrightarrow(e^{i\psi}e^{i\theta},e^{i\theta},e^{-i\psi}e^{-2i\theta}).\]
 
We want to find the possible $\G$-homogeneous bundles over the singular orbit which are extensions of the bundles on $\R^+\times\SU(2)^2/\Z_4$ found in Section \ref{sec:eveqmn}. Such bundles are parameterised by homomorphisms $K\to\SU(2)$, which up to conjugacy are of the form \[\lambda_s^{\nu,\mu}(e^{i\psi},e^{i\theta})=\left(\begin{array}{cc} e^{i(\nu\psi+\mu\theta)} & 0\\ 0 & e^{-i(\nu\psi+\mu\theta)} \end{array}\right),\] 
for $\nu,\mu\in\Z$, such that $\lambda_s^{\nu,\mu}\circ\iota=\lambda_p^{l,1}$ for $l\in\{0,1\}$; here, $\iota:K_0\to K$ is the map given by the obvious inclusion $\Z_2\hookrightarrow U(1)$ on the first factor and the identity on the second. Then we must have $\mu=1$ while $\nu\equiv l\mod 2$. 

To apply the results of Appendix \ref{sec:invtens}, it is easier to consider extending smoothly the $\SU(2)^2$-homogeneous bundles over the singular orbit and to forget the extra U(1)-action. For $j\in\Z$, the $\SU(2)^2$-homogeneous bundles $P_s^j$ over $\SU(2)^2/K_{1,1}$ correspond to isotropy homomorphisms $\lambda_s^j(e^{i\theta})=\text{diag}(e^{ij\theta},e^{-ij\theta})$. The bundle isomorphism induced by \eqref{eq:isotoU(1)} extends to the singular orbit, with $P_s^{\nu,1}$ isomorphic to $P_s^j$ if $j=2\nu-1$. Then we must have $j$ odd, where $P_s^j$ is an extension of the bundles $P_p^{\pm1}$ over the principal orbits if $j\equiv \pm1\mod 4$. Then as in Section \ref{sec:hombund}, we have bundles $P_j$ over $M_{m,n}$ whose restriction to the principal orbits is $P_p^{\pm1}$ and whose restriction to the singular orbit is $P_s^j$. 

Using $A^{\text{can}}=jE_3\otimes \tfrac{1}{2}(e_3-e_3')$ as a reference connection, we can write any $\SU(2)^2$-invariant connection as an $\SU(2)^2$-invariant element of $\Omega^1(\Ad P_j)$. In what follows, we apply the results of Appendix \ref{sec:invtens} to the 1-form 
\[A-A^{\text{can}}=f(E_1\otimes e_1+E_2\otimes e_2)+f'(E_1\otimes e_1'+E_2\otimes e_2')+(g-j)(E_3\otimes \tfrac{1}{2}(e_3-e_3'))+h(E_3\otimes(e_3+e_3')).\]
We now state Lemma \ref{lemma:nonabmn} of Appendix \ref{sec:invtens} in the case where $m=n=1$ and $r=s=\frac{1}{2}$, which gives the conditions on the coefficients of $A-A^{\text{can}}$ such that $A$ extends smoothly over the singular orbit.

\begin{lemma}
\label{lemma:nonab11}
Consider the manifold $M_{1,1}$ and let $j\in\Z$. Then the $\SU(2)^2$-invariant connection over the principal orbits
\[A  = A_{12}(E_1\otimes e_1+E_2\otimes e_2)+ A_{12}'(E_1\otimes e'_1+E_2\otimes e'_2)+A_{rs}(E_3\otimes \tfrac{1}{2}(e_3-e_3'))+A_{11}(E_3\otimes(e_3+e_3'))\]
on $P_j$ with $j\equiv \pm1\mod4$ extends to a smooth section of $\Omega^1(\Ad P_j)$ if and only if $A_{rs},A_{11}$ are even, with $A_{rs}(0)=0$, and we have
\[\begin{cases}A_{12} \text{ even and } O(t^{\left\vert\frac{j-1}{2}\right\vert}), A_{12}' \text{ odd and } O(t^{\left\vert\frac{j+1}{2}\right\vert}) & \text{ for }j\equiv 1 \mod 4; \\
A_{12} \text{ odd and } O(t^{\left\vert\frac{j-1}{2}\right\vert}), A_{12}' \text{ even and } O(t^{\left\vert\frac{j+1}{2}\right\vert}) & \text{ for }j\equiv -1 \mod 4.
\end{cases}\]
\end{lemma}

In Section \ref{sec:triv}, we find all global abelian solutions, so we are only concerned here with finding local solutions which have full gauge group $\SU(2)$. Recall that we need only consider odd $j\geq1$ because of the bundle isomorphism $P_{j}\to P_{-j}$ induced by the involution that swaps the two factors of $\SU(2)$. The following proposition describes solutions to the $G_2$-instanton equations in a neighbourhood of the singular orbit on $P_j$ for positive odd $j$.

\begin{proposition}
\label{prop:11j}
Let $Y\subset M_{1,1}$ contain the singular orbit $S^3\times S^3/K_{1,1}$, and be endowed with an \newline $\G$-invariant AC or ALC holonomy $G_2$-metric of \cite{FHN18}. For $j=2\nu-1$ with $\nu\in\Z_+$, there is a 2-parameter family of $\G$-invariant $G_2$-instantons $A_j$ on $P_j$, parameterised by $f_0,h_0\in\R$, with gauge group $\SU(2)$ in a neighbourhood of the singular orbit in $Y$, extending smoothly over the singular orbit.
Furthermore, $A_j$ can be written as in Proposition \ref{prop:11ODEs}, where $f$, $f'$, $g$ and $h$ satisfy,
        \begin{align*}
        f(t) & = f_0t^{\nu-1} + O(t^{\nu+1}), \\
        f'(t) & = \frac{\beta^3(1-2h_0)+1-\nu}{4\nu r_0\beta^2}f_0t^{\nu}+O(t^{\nu+2}),\\
        g(t) &=2\nu-1+O(t^2), \\
        h(t)&=h_0+O(t^2).
        \end{align*}
\end{proposition}
\begin{proof}
By Lemma \ref{lemma:nonab11}, we can write
\[f(t)=u_f(t)t^{\nu-1}, \hspace{1cm} f'(t)=u_{f'}(t)t^{\nu}, \hspace{1cm} g(t)=u_g(t),\hspace{1cm} h(t)=u_h(t).\]
The vanishing at $t=0$ of the coefficient of $\frac{1}{2}(e_3-e_3')$ in $A-A^{\text{can}}$ forces $u_g(0)=j=2\nu-1$. Then the $G_2$-instanton equations in a neighbourhood of the singular orbit take the form
\begin{align*}
    \dot{u}_f(t)&=\left(\tfrac{1}{2}(u_g(t)+1)-\nu\right)u_f(t)t^{-1} +O(1),\\
    \dot{u}_{f'}(t)&= \left((-\tfrac{1}{2}(u_g(t)+1)-\nu)u_{f'}(t)+\frac{1-u_g(t)+2\beta^3(1-2u_h(t))}{4r_0\beta^2}u_f(t)\right)t^{-1} +O(1),\\
    \dot{u}_g(t) & = O(t),\\
    \dot{u}_h(t) & = O(t).
\end{align*}
Thus the system has the form of the IVP in Theorem \ref{thm:4.3} if we let $X(t)=(u_f(t),u_{f'}(t),u_g(t),u_h(t))$ and 
\[M_{-1}(X) = \left(\left(\tfrac{1}{2}(u_g+1)-\nu\right)u_f,\left((-\tfrac{1}{2}(u_g+1)-\nu)u_{f'}+\frac{1-u_g+2\beta^3(1-2u_h)}{4r_0\beta^2}u_f\right),0,0\right).\]
Then setting $M_{-1}(X(0))=0$ we have, for $f_0,h_0\in\R$,
\begin{align*}
    u_f(0)&=f_0 \\
    u_{f'}(0) & = \frac{\beta^3(1-2h_0)+1-\nu}{4\nu r_0\beta^2}f_0,\\
    u_g(0) & = 2\nu-1,\\
    u_h(0) & = h_0.
\end{align*}
We have $\det(h\text{Id}-dM_{-1}(X(0)))=h^3(h+2\nu)\neq 0$ for $h\in\mathbb{N}_{\geq1}$, thus Theorem \ref{thm:4.3} gives a real analytic solution to the singular IVP. Then there is a 2-parameter family of solutions, which are parameterised by $h_0$ and $f_0$, as in the statement of the proposition.
\end{proof}

\section{Solutions}\label{sec:sols}

We now outline some elementary solutions to the $\G$-invariant $G_2$-instanton equations. We then apply a dynamical systems approach to find solutions on $M_{1,1}$ with full gauge group. These solutions are constructed by flowing from a perturbation of an abelian solution close to the singular orbit to a perturbation of a dilation-invariant solution.

\subsection{Elementary Solutions}\label{sec:triv}

We now consider global solutions to the $G_2$-instantons which extend the local solutions of the previous section across the whole of $M_{1,1}$. We start by considering elementary solutions which are either flat or abelian. The former are not gauge equivalent to the trivial connection under invariant gauge transformations; such transformations $g:\R^+\to\SU(2)$ are constant on each principal orbit. The latter are connections whose gauge group reduces to U(1); as the system of ODEs decouples, we can give a 1-parameter family of explicit abelian solutions on each bundle $P_j$. One special member of this family will prove useful in Section \ref{sec:dyn} for finding solutions with full gauge group $\SU(2)$.

\subsubsection{Invariant flat connections}

Invariant flat connections $A=\alpha(t)$ are given by constant choices of values for the functions $f,f',g,h$. Then $\dot{\alpha}=0$ and hence $F_A=dt\wedge \dot{\alpha}+F_{\alpha(t)}=F_{\alpha(t)}$. So we require $F_{\alpha(t)}=0$ on each principal orbit.

\begin{proposition}
\label{prop:flatsols11}
The connections on the bundle $P_1$ over $M_{1,1}$, given by the following connection forms
\[A_1^{\pm} = \pm(E_1\otimes e_1 + E_2\otimes e_2) +E_3\otimes e_3\]
are flat $G_2$-instantons.
\end{proposition}
\begin{proof}
We can take the constant solutions $(f,f',g,h) = (\pm1,0,1,\frac{1}{2})$ to the ODE system \eqref{eq:G2ode11}. Subsituting the corresponding values of $\alpha_i$, $\alpha_i'$ into \eqref{eqn:curvature} yields $F_{\alpha}=0$. The resulting connection forms are $A_1^{\pm}$.
\end{proof}

It is easy to show that the instantons of Proposition \ref{prop:flatsols11}, together with the canonical invariant connection, are the only invariant flat instantons on $M_{1,1}$. 

\begin{remark}\label{rem:gaugetrans}
We note that some of these solutions are gauge equivalent under an $\G$-invariant gauge transformation. Indeed, the element
\[\left(\begin{array}{cc} i & 0 \\ 0 & -i \end{array}\right)\in\SU(2)\]
defines a group action of $\SU(2)$ on itself by conjugation. The derivative of this action is the adjoint representation, which acts on $\su(2)$ by
\[E_1\mapsto -E_1, \hspace{1cm} E_2\mapsto -E_2, \hspace{1cm} E_3\mapsto E_3.\]
The $G_2$-instanton equations \eqref{eq:G2ode11} are invariant under this gauge transformation and $A^+_1$ and $A^-_1$ are gauge equivalent.
\end{remark}

\begin{remark}
The instantons of Proposition \ref{prop:flatsols11} do not depend on the functions $a$ and $b$, since the coefficients of $\Phi_1$, $\chi_1$ and $\Psi_1$ vanish independently of each other. Hence they also define invariant flat $G_2$-instantons on $M_{1,1}$ equipped with the torsion free ALC $G_2$-structures given in case (v) of Theorem \ref{thm:class}.
\end{remark}

\subsubsection{Abelian solutions}

When $M$ is a compact $G_2$-manifold, every class in $H^2(M;\R)$ has a harmonic representative; when $M$ is non-compact, this is not necessarily true. Proposition 4.57 of \cite{KL19} states that the space of $L^2$-integrable closed and coclosed 2-forms on $M$ is isomorphic to $H_{cs}^2(M;\R)$.
Poincar\'e Duality tells us that for $M_{1,1}$, this space is 1-dimensional and so we would expect to find a unique $L^2$-integrable abelian instanton $A$ on each U(1)-bundle on $M_{1,1}$.

Abelian solutions are given by morphisms $\Lambda:(\R\oplus\C_{2})\oplus(\R\oplus\C_{-2})\to\R$ of $\Z_{4}$-representations; such maps must vanish on the complex components, corresponding to $f=f'=0$ in \eqref{eq:A}. We start by noticing that the $G_2$-instanton equations \eqref{eq:G2ode11} with vanishing $f$ and $f'$ are uncoupled. Hence, we solve them in terms of the functions $a$ and $b$, with initial conditions given by Proposition \ref{prop:11j}. The general solutions for $g$ and $h$ are
\begin{align}
\label{eqn:11gab}
g(t)&=\frac{4jr_0^6}{(b+r_0^3)^2},   \\
\label{eqn:11hab}
h(t) &= h_0\exp\left(\int_0^t\frac{2\dot{b}(b-r_0^3)}{4a^2-(b-r_0^3)^2}d\tau\right).
\end{align}

By taking $h_0=0$ on $P_j$, we get a particular solution
\[\Aab=\frac{1}{2}gE_3\otimes(e_3-e_3')\]
with $L^2$-bounded curvature, where $g$ is given by (\ref{eqn:11gab}). This abelian instanton will play a role in the construction of $G_2$-instantons with gauge group $\SU(2)$. For this choice of $h_0$, $g$ decays at infinity so the connection $\Aab$ is asymptotic to the flat instanton on $P_j$ known as the canonical invariant connection, as defined in Section \ref{sec:hombund}.

\subsection{Solutions via a Dynamical Systems Approach}\label{sec:dyn}

We now construct solutions to the $G_2$-instanton equations which have full gauge group $\SU(2)$ and are perturbations of an abelian solution near the singular orbit. We begin by rescaling time so that, using the expansions of $a$ and $b$ given in Proposition \ref{prop:5.3}, we can rewrite  \eqref{eq:G2ode11} as a sum of autonomonous and non-autonomous parts. Denote $z=(f,f',g,h)$. Rescaling time as $t(\tau)=\exp(\tau)$ yields $\frac{dz}{d\tau} = \dot{z} \frac{dt}{d\tau} = \dot{z} \exp(\tau)$; hence, we can write
\begin{equation}\label{eq:defF}
  \frac{dz}{d\tau} =  F(z)+G(z,\tau),
\end{equation}
where
\begin{equation}\label{eq:F}
  F(z)=
  \begin{pmatrix}
          2 f'(2-3h+\frac{1}{2}g) + 2 f(g-1) \\
         2  f(2-3h-\frac{1}{2}g) - 2 f' (g+1) \\
         6 (f^2-(f')^2-g) \\
         2h -  (f')^2 -f^2 -4 f f'
       \end{pmatrix}
\end{equation}
and
\begin{equation}\label{eq:G}
  G(z,\tau)= 36\sqrt{3}r_0^3(\exp(\tau))^{-3}
  \begin{pmatrix}
          f'(h-1-\frac{1}{2}g) \\
         f(h+\frac{1}{2}g-1) \\
         6 ((f')^2-f^2+g) \\
          \frac{1}{2} (f')^2 +\frac{1}{2} f^2-h
       \end{pmatrix}+O(\exp(\tau)^{-6}).
\end{equation}
Then the non-autonomous part of \eqref{eq:defF} satisfies
\begin{align}\label{eq:Glim}
   & \lim_{\tau \to \infty } \exp(\tau) G(z,\tau) =0  \\
   & \lim_{\tau \to \infty } \exp(\tau) D_1G(z,\tau) =0.\label{eq:DGlim}
\end{align}

\begin{remark}
The properties \eqref{eq:Glim} and \eqref{eq:DGlim} no longer hold when we consider the ALC members of the $\Csev$ family, since the autonomous system corresponds to instantons on the cone. This means that the results which follow cannot be extended outside the AC limit.
\end{remark}

\subsubsection{Autonomous Dynamics and Steady States}

We start by analysing the dynamical systems behaviour of the truncated autonomous ODE
\begin{equation}\label{eq:ODEF}
  \dot{z}=F(z)
\end{equation}
with flow $\Phi(\cdot,\tau)$. Subsequently, we then show how certain solutions of the autonomous ODE persist in the full system \eqref{eq:defF}.

It is easy to see that \eqref{eq:ODEF} has steady states
\begin{gather*}
  z_0=(0,0,0,0), \quad z_+=\left(\frac{1}{3},\frac{1}{3},0,\frac{1}{3}\right), \quad z_-=\left(-\frac{1}{3},-\frac{1}{3},0,\frac{1}{3}\right).
\end{gather*}
The fixed points $z_+$ and $z_-$ are interchanged by the gauge transformation described in Remark \ref{rem:gaugetrans}. We calculate the respective linearisations at the steady states and their eigenvalues and eigenvectors; the results are given in Table \ref{tab:fix}.

\begin{table}[h]
    \centering
\begin{tabular}{|l|l|l|l|}
\hline
Fixed Point $z$ & $DF(z)$ & Eigenvalues & Eigenvectors \\ \hline
       $z_0$         &    $\left( \begin {array}{cccc} -2&4&0&0\\ \noalign{\medskip}4&-2&0&0
\\ \noalign{\medskip}0&0&-6&0\\ \noalign{\medskip}0&0&0&2
\end {array} \right)$     &   $\{-6,-6,2,2 \}$          &   $\left( \begin {array}{cccc} 0&-1&0&1\\ \noalign{\medskip}0&1&0&1
\\ \noalign{\medskip}1&0&0&0
\\ \noalign{\medskip}0&0&1&0\end {array} \right)$           \\ \hline
         $z_+$       &    $\left( \begin {array}{cccc} -2&2&1&-2\\ \noalign{\medskip}2&-2&-1&-2\\ \noalign{\medskip}4&-4&-6&0\\ \noalign{\medskip}-2&-2&0&2\end {array} \right)$     &    $\{4,-8,-2,-2 \}$         &           $\left( \begin {array}{cccc} -1&-1&1&1\\ \noalign{\medskip}-1&1&1&-1
\\ \noalign{\medskip}0&4&0&2
\\ \noalign{\medskip}2&0&1&0\end {array} \right)$   \\ \hline
         $z_-$       &  $\left( \begin {array}{cccc} -2&2&-1&2\\ \noalign{\medskip}2&-2&1&2\\ \noalign{\medskip}-4&4&-6&0\\ \noalign{\medskip}2&2&0&2\end {array} \right)$       &      $\{4,-8,-2,-2 \}$       &        $\left(\begin {array}{cccc} 1&1&-1
&-1\\ \noalign{\medskip}1&-1&-1&1
\\ \noalign{\medskip}0&4&0&2
\\ \noalign{\medskip}2&0&1&0\end {array} \right)$      \\ \hline
\end{tabular}
\caption{The linearisation, eigenvalues and corresponding matrix of eigenvectors for each fixed point $z_0$, $z_{\pm}$.}
\label{tab:fix}
\end{table}

The sets of eigenvalues together show that all steady states are hyperbolic saddles, which implies the existence of respective stable and unstable manifolds, tangential to the stable and unstable eigenspaces. For $z_0$, both the stable and unstable manifolds are two dimensional. For $z_+$ and $z_-$, the stable manifold is three dimensional and the unstable manifold is one dimensional.

\subsubsection{Constituent solutions as building blocks}\label{sec:build}

We now consider a couple of solutions to the full system which we use as building blocks for a gluing process, to form solutions which initially move towards the fixed point $z_0$ before travelling off to one of the fixed points $z_{\pm}$. We start by describing the subspaces of the 4-dimensional phase space which are invariant under the flow, and the solutions which lie inside these invariant spaces.

The plane $\Pi_1=\{f=f'=0\}$ is invariant under both \eqref{eq:defF} and \eqref{eq:ODEF}. Solutions which live in this plane for all time are abelian and fit into the family given by \eqref{eqn:11gab} and \eqref{eqn:11hab}; recall that for each bundle $P_j$, we have a solution \begin{equation}\label{eq:Aab}\Aab=\frac{2jr_0^6}{(b+r_0^3)^2}E_3\otimes(e_3-e_3').\end{equation}
This solution traces out a segment of the invariant line $\ell_1=\text{span}\{(0,0,1,0)\}$, which is the stable manifold of the saddle point $(0,0)$ of the reduced system in $\Pi_1$. We have $g(\tau) \to 0 \mbox{ as } \tau \to \infty$ and we can find a time $T$ such that $|g(T)| <\delta$ for $\delta>0$.

\begin{remark}
When $m$ and $n$ are distinct integers, the line $\ell_1=\text{span}\{(0,0,m+n,r-s)\}\subset\Pi_1$ is not invariant under \eqref{eq:defF}. Together with the absence of an explicit abelian solution, these are the reasons that we cannot immediately apply what follows to $M_{m,n}$ for general positive coprime $m$ and $n$. 
\end{remark}

Another invariant plane of both \eqref{eq:defF} and \eqref{eq:ODEF} is $\Pi_2=\{f=f',g=0\}$. In this plane, the autonomous equation \eqref{eq:ODEF} reduces to
\begin{align*}
  \dot{f}=\dot{f'} & = 2 f - 6h f, \\
  \dot{g} & =0, \\
  \dot{h} & = 2 h -6 f^2.
\end{align*}
We note that $\Pi_2$ contains the steady states $z_0$ and $z_{\pm}$, as well as all of their unstable manifolds. Within this plane, the lines
$\ell_{\pm}=\{f=f'= \pm h\}$ are also invariant under \eqref{eq:ODEF}, which then reduces to the single equation
\begin{equation}\label{eq:lineode}
  \dot{h}= 2h - 6 h^2.
\end{equation}
There is an explicit solution
\[h(\tau) = \frac{\exp(2\tau)}{1+3\exp(2\tau)}\]
and up to time translation, this is the unique solution with image $(0,\frac{1}{3})$.
It yields a solution for \eqref{eq:ODEF}, namely
 \begin{equation} \label{eq:hode}
   z(\tau)= (f,f',g,h)(\tau)  = \left(\pm \frac{\exp(2\tau)}{1 + 3 \exp(2\tau)},\pm \frac{\exp(2\tau)}{1 + 3 \exp(2\tau)}, 0 , \frac{\exp(2\tau)}{1 + 3 \exp(2\tau)}\right).
 \end{equation}
Up to time translation, this is the unique solution lying in $\ell_{\pm}$ that satisfies
\[z(\tau)\xrightarrow{\tau\to-\infty} z_0 \text{ and } z(\tau)\xrightarrow{\tau\to\infty} z_{\pm}.\]
It also lies in the intersections $W^u(z_0) \cap W^s(z_+)$  and $W^u(z_0) \cap W^s(z_-)$ respectively. For $\tau \to \infty$, the stable manifolds $W^s(z_\pm)$ are tangential to their respective stable eigenspaces; each of these spaces, together with the unstable eigenvector contained in the plane $\Pi_2$, span the whole of $\R^4$. At any point $z(\tau)$ in \eqref{eq:hode}, \[T_{z(\tau)}W^u(z_0)=\Pi_2,\] 
while the tangent space of $W^s(z_{\pm})$ at $z_{\pm}$ is the direct sum of $\ell_{\pm}$ and a plane transverse to $\Pi_2$. If the tangent space to $W^s(z_{\pm})$ contains $\Pi_2$ at any point on the flow line, then it will contain $\Pi_2$ everywhere along the flow line. With these observations, we see that the intersections $W^u(z_0) \cap W^s(z_+)$  and $W^u(z_0) \cap W^s(z_-)$ are transversal along $z(\tau)$ for all $\tau$. 
 
We now look to perturb the solutions \eqref{eq:hode} of the autonomous system to solutions of the full non-autonomous system. In this setting, the stable manifold now depends on the choice of initial time; let $\phi^{\tau_0}(\tau,z_0)$ be the associated flow to the ODE \eqref{eq:defF} with $z(\tau_0)=z_0$. For each $\tau_0$, we have a stable manifold \[W^s_{\tau_0}(z_{\pm})=\{z_0\in\R: \lim_{\tau\to\infty}\phi^{\tau_0}(\tau,z_0)= z_{\pm}\}.\] In order to perturb the solutions of the autononomous system, we show that such a perturbation will not affect the existence of the stable manifold of $z_{\pm}$ for sufficiently large start times $\tau_0=T$ and that these perturbed stable manifolds will still intersect the unstable manifold of $z_0$ in the same way for large $T$.

The construction of the local stable manifolds from a contraction mapping argument in Theorem 9.3 of \cite{T12} shows that small enough perturbations will not change the existence of the stable manifold. Indeed, adding the non-autonomous function $G$ to the existing contraction still yields a contraction for a sufficiently large start time; see Appendix \ref{sec:dynsys} for more details. This contraction will have a unique fixed point that defines the stable manifold which will still be tangent to the same stable eigenspace.

From the local stable manifolds of $z_\pm$, it only takes a finite time to reach the intersection points constructed above. By smooth dependence on initial conditions and perturbations, we have that the perturbed stable manifolds $W^s_{\tau_0}(z_\pm)$ are $C^1$-close to the original stable manifolds near $z_{\pm}$ for all large $\tau_0=T$, due to \eqref{eq:Glim} and \eqref{eq:DGlim}.

\begin{remark}
The solutions \eqref{eq:hode} solve the $G_2$-instanton equations on the $G_2$-cone, when we translate back into the original time. We remark that it is feasible that a global solution could be constructed by gluing the pull-back of this solution on the cone by the diffeomorphism giving the AC structure to a perturbation of the abelian solution \eqref{eq:Aab}. However, we instead apply in the following section a dynamical systems argument to obtain such a global solution.
\end{remark}

\subsubsection{Global solutions which extend smoothly over the singular orbit} \label{ssec:glue}

We now state the main theorem which gives the existence of solutions to the $G_2$-instanton equations with gauge group $\SU(2)$ that are formed using the building blocks from Section \ref{sec:build}.

\begin{theorem}\label{thm:main}
Let $M_{1,1}$ be the AC $G_2$-manifold of Theorem \ref{thm:7.1} and let $\Aab$ be the unique abelian solution \eqref{eq:Aab} on each bundle $P_j$ over $M_{1,1}$ for $j\in\Z$ odd. There is a function $(-\epsilon,\epsilon)\to\R,f_0\mapsto h_0(f_0)$ such that the unique local solution with initial condition $(f_0, h_0(f_0))$ given by Proposition \ref{prop:11j}, which is a small perturbation of $\Aab$ near the singular orbit $S^2\times S^3$, extends to a global solution. Thus we obtain a 1-parameter family of $\G$-invariant $G_2$-instantons with full gauge group $\textup{SU}(2)$ and bounded curvature. These instantons are asymptotic to the pull-back to the cone of the nearly K\"ahler instanton on $S^3\times S^3/\Z_4$.
\end{theorem}

We prove this theorem via the following argument. \vspace{-0.4cm}
\begin{enumerate}
\item[\textup{(i)}] We show that a 2-parameter family of initial conditions is mapped at some time $T_1$ under the flow to a disc that intersects the stable manifold of $z_0$ transversely, by considering the linearisation of the ODE system around the abelian solution $\Aab$. 
\item[\textup{(ii)}] We also show that such a disc is mapped at some time $T_2>T_1$ under the flow to be $C^1$ close to the unstable manifold of $z_0$. Since we showed that this intersection is 1-dimensional in Section \ref{sec:build}, we see that the image of this disc under the flow intersects the stable manifolds of $z_{\pm}$ along a curve. 
\item[\textup{(iii)}] Together this yields a 1-parameter family of global solutions.
\end{enumerate}

In what follows, we choose to swap the order in which we prove the theorem. Stage 1 will address part (ii) of the argument on the interval $(T_2,\infty)$ while Stage 2 concerns part (i) on the interval $(0,T_1)$. We will use the following schematic diagram, shown in Figure \ref{fig:schemeintro}, to illustrate the components of the solutions (iii).

\tikzset{
  mynode/.style={fill,circle,inner sep=2pt,outer sep=0pt}
}

\begin{figure}[h]
\centering
\begin{tikzpicture}
\draw[black,thick] (0,0) -- (6,0)
    node[pos=0,mynode,fill=black,label=below:\textcolor{black}{Initial Conditions}]{}
    node[pos=1,mynode,fill=black,text=blue, label=below:\textcolor{black}{$z_0$}]{}
    node[pos=0.5,label=below:\textcolor{black}{$\Aab$}]{};
    \draw[black,thick] (6,0) -- (10,2)
    node[pos=1,mynode,fill=black,text=blue, label=below:\textcolor{black}{$z_+$}]{}
    node[pos=0.5,label=below:\textcolor{black}{$\ell_+$}]{};
    \draw[black,thick] (6,0) -- (10,-2)
    node[pos=1,mynode,fill=black,text=blue, label=below:\textcolor{black}{$z_-$}]{}
    node[pos=0.5,label=below:\textcolor{black}{$\ell_-$}]{};
    \draw[black,thick,->] (0,-3) -- (10,-3)
    node[pos=0,mynode,fill=black,text=blue, label=below:\textcolor{black}{$t=0$}]{}
    node[pos=0.4,mynode,label=below:\textcolor{black}{$T_1$}]{}
    node[pos=0.55,mynode,label=below:\textcolor{black}{$T_2$}]{};
    \draw [purple] plot [smooth, tension=0.3] coordinates { (0,0.2) (6,.3) (10,2)};
\end{tikzpicture}
\caption{Schematic diagram showing the components of the 1-parameter family of solutions of Theorem \ref{thm:main}.}\label{fig:schemeintro}
\end{figure}
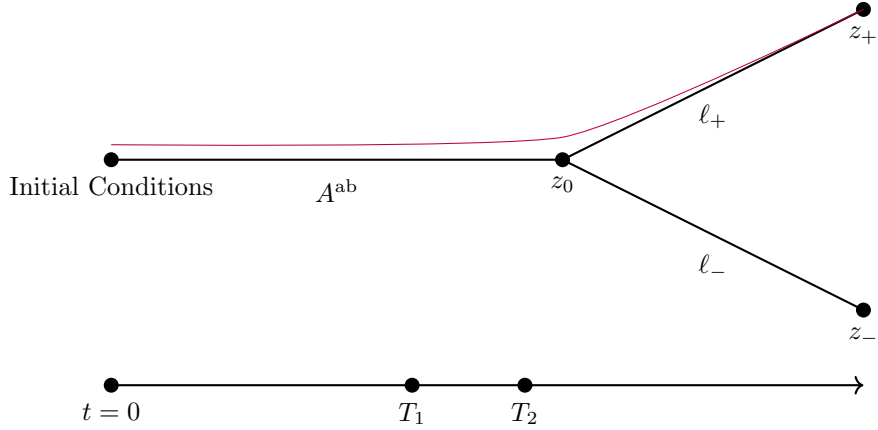

The diagram shows a solution (in purple) tending to $z_+$ as $t\to\infty$, evolving with time from left to right. The first part of the solution is shown to be a perturbation of the abelian solution $\Aab$, shown in the diagram as a straight line from the initial conditions to $z_0$. The solution is seen to bypass the point $z_0$ before tending to the fixed point $z_+$. 

\subsubsection{Stage 1}\label{sec:stage1}

As before, we only consider the case where $j\geq1$. We show that there are solutions to \eqref{eq:defF}, which initially closely follow \eqref{eq:Aab} and then travel along perturbed solutions of \eqref{eq:hode}. We appeal to the Inclination Lemma, which gives a convergence result for transversal manifolds.

\begin{lemma}[The Inclination Lemma, \cite{MP12} Lemma 7.1] \label{lemma:inc}
Let $B^s, B^u$ be balls contained in the local stable and unstable manifolds, respectively, of a hyperbolic fixed point $0$; set $V=B^s\times B^u$. Consider a point $q$ in the local stable manifold, and a disc $D^u$ of the same dimension as the local unstable manifold which is transversal to the local stable manifold at the point $q$. Let $D^u_t$ be the connected component of $V\cap \Phi(t, D^u)$ to which $\Phi(t,q)$ belongs. Given $\epsilon>0$, there exists $T\in\R$ such that if $t>T$, then $D^u_t$ is $\epsilon$ $C^1$-close to $B^u$.
\end{lemma}

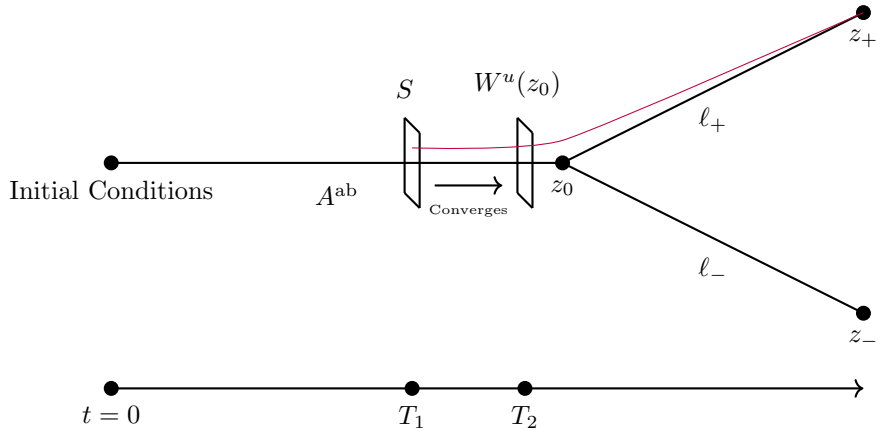
\begin{figure}[h]
\centering
\begin{tikzpicture}
\draw[black,thick] (0,0) -- (6,0)
    node[pos=0,mynode,fill=black,label=below:\textcolor{black}{Initial Conditions}]{}
    node[pos=1,mynode,fill=black,text=blue, label=below:\textcolor{black}{$z_0$}]{}
    node[pos=0.5,label=below:\textcolor{black}{$\Aab$}]{};
    \draw[black,thick] (6,0) -- (10,2)
    node[pos=1,mynode,fill=black,text=blue, label=below:\textcolor{black}{$z_+$}]{}
    node[pos=0.5,label=below:\textcolor{black}{$\ell_+$}]{};
    \draw[black,thick] (6,0) -- (10,-2)
    node[pos=1,mynode,fill=black,text=blue, label=below:\textcolor{black}{$z_-$}]{}
    node[pos=0.5,label=below:\textcolor{black}{$\ell_-$}]{};
    \draw[black,thick,->] (0,-3) -- (10,-3)
    node[pos=0,mynode,fill=black,text=blue, label=below:\textcolor{black}{$t=0$}]{}
    node[pos=0.4,mynode,label=below:\textcolor{black}{$T_1$}]{}
    node[pos=0.55,mynode,label=below:\textcolor{black}{$T_2$}]{};
    \draw[black,thick] (3.9,-0.4) -- (3.9,0.6)
    node[pos=1,label=above:\textcolor{black}{$S$}]{};
    \draw[black,thick] (4.1,-0.6) -- (4.1,0.4);
    \draw[black,thick] (3.9,-0.4) -- (4.1,-0.6);
    \draw[black,thick] (3.9,0.6) -- (4.1,0.4);
    \draw[black,thick] (5.4,-0.4) -- (5.4,0.6)
    node[pos=1,label=above:\textcolor{black}{$W^u(z_0)$}]{};
    \draw[black,thick] (5.6,-0.6) -- (5.6,0.4);
    \draw[black,thick] (5.4,-0.4) -- (5.6,-0.6);
    \draw[black,thick] (5.4,0.6) -- (5.6,0.4);
    \draw[black,thick,->] (4.3,-0.3) -- (5.2,-0.3)
    node[pos=0.5,label=below:\textcolor{black}{\tiny Converges}]{};
    \draw [purple] plot [smooth, tension=0.3] coordinates { (4,0.2) (6,.3) (10,2)};
\end{tikzpicture}
\caption{Schematic diagram showing the transverse manifold $S$ converging to the unstable manifold of $z_0$.}\label{fig:scheme1}
\end{figure}

By Lemma \ref{lemma:inc}, we know that the forward solution to the autonomous system \eqref{eq:ODEF} of any transversal manifold $S$ to $W^s(z_0)$ at a point in $ \ell_1\subset W^s(z_0)$ at $t=T_1$ converges locally in the $C^1$ topology to $W^u(z_0)$ for large times $t>T_2$. This is shown schematically in Figure \ref{fig:scheme1}.

Since the perturbation $G$ in \eqref{eq:defF} is sufficiently small, the arguments of Section \ref{sec:build} apply, and we see that $S$ still intersects the stable manifolds of $z_\pm$ for the full system at sufficiently large times. Then all intersection points have the desired property that the forward solutions converge to $z_\pm$.

Essentially, we see that as $S$ evolves with the flow, its intersection with each of the stable manifolds $W^s(z_{\pm})$ is a 1-dimensional subset; this subset tends to a segment of the corresponding line $\ell_{\pm}$ as $\tau\to\infty$. With this in mind, we now transform back to the original time $t$.

\subsubsection{Stage 2}\label{sec:stage2}

The final step is to show that the 2-parameter family of initial conditions which gives a smooth extension over the singular orbit at time $t=0$ is mapped at time $t=T$ to a 2-dimensional submanifold $S$ which is transverse to the stable manifold of $z_0$, as shown in Figure \ref{fig:scheme2}. 

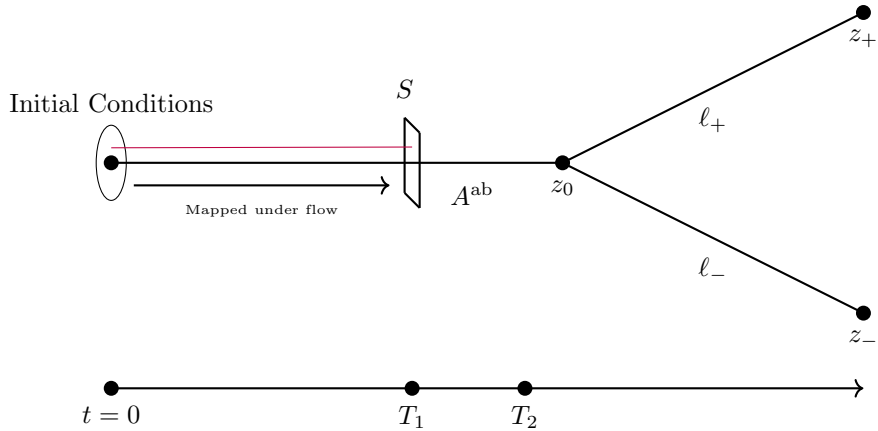
\begin{figure}[h]
\centering
\begin{tikzpicture}
\draw[black,thick] (0,0) -- (6,0)
    node[pos=0,mynode,fill=black]{}
    node[pos=1,mynode,fill=black,text=blue, label=below:\textcolor{black}{$z_0$}]{}
    node[pos=0.8,label=below:\textcolor{black}{$\Aab$}]{};
    \draw[black,thick] (6,0) -- (10,2)
    node[pos=1,mynode,fill=black,text=blue, label=below:\textcolor{black}{$z_+$}]{}
    node[pos=0.5,label=below:\textcolor{black}{$\ell_+$}]{};
    \draw[black,thick] (6,0) -- (10,-2)
    node[pos=1,mynode,fill=black,text=blue, label=below:\textcolor{black}{$z_-$}]{}
    node[pos=0.5,label=below:\textcolor{black}{$\ell_-$}]{};
    \draw[black,thick,->] (0,-3) -- (10,-3)
    node[pos=0,mynode,fill=black,text=blue, label=below:\textcolor{black}{$t=0$}]{}
    node[pos=0.4,mynode,label=below:\textcolor{black}{$T_1$}]{}
    node[pos=0.55,mynode,label=below:\textcolor{black}{$T_2$}]{};
    \draw[black,thick] (3.9,-0.4) -- (3.9,0.6)
    node[pos=1,label=above:\textcolor{black}{$S$}]{};
    \draw[black,thick] (4.1,-0.6) -- (4.1,0.4);
    \draw[black,thick] (3.9,-0.4) -- (4.1,-0.6);
    \draw[black,thick] (3.9,0.6) -- (4.1,0.4);
    \draw[black,thick,->] (0.3,-0.3) -- (3.7,-0.3)
    node[pos=0.5,label=below:\textcolor{black}{\tiny Mapped under flow}]{};
    \draw (0,0) ellipse (0.2cm and 0.5cm); 
    \node at (0,0.8) {Initial Conditions};
    \draw [purple] plot [smooth, tension=0.3] coordinates { (0,0.2) (4,.21)};
\end{tikzpicture}
\caption{Schematic diagram showing a 2-dimensional subspace of initial conditions flowing to a 2-dimensional transverse manifold at $t=T_1$.}\label{fig:scheme2}
\end{figure}

To study perturbations with respect to the initial data, we linearise the system around $\Aab$; for notational ease, we write
\[\widetilde{\Phi}_1:=\frac{\Phi_1}{2\dot{a}^3\dot{b}^2} \hspace{1cm} \text{and} \hspace{1cm} \widehat{\Phi}_1 := \frac{\Phi_1}{2\dot{a}^4\dot{b}}\]
and the same for $\Psi_1$ and $\chi_1$.

The linearisation of the ODEs at $f=f'=h=0$ is
\[
A(t)\!=\!D_2(F+G)(t,(0,0,g,0))\!=\!\left( \small\begin {array}{cccc} \widetilde{\chi}_1(g-1)&\widetilde{\Phi}_1+\frac{1}{2}g(\widetilde{\Phi}_1+\widetilde{\Psi}_1) &0&0\\ \noalign{\medskip}\widetilde{\Phi}_1-\frac{1}{2}g(\widetilde{\Phi}_1+\widetilde{\Psi}_1)&-\widetilde{\chi}_1(g+1)&0&0
\\ \noalign{\medskip}0&0&\widehat{\Psi}_1- \widehat{\Phi}_1&0\\ \noalign{\medskip}0&0&0&\widehat{\Phi}_1+ \widehat{\Psi}_1
\end {array} \right).  \]

The space of possible initial conditions on each bundle $P_j$ is 2-dimensional, parametrised by $f_0$ and $h_0$, where such local solutions are given by the expansions given in Proposition \ref{prop:11j}, namely for $j=2\nu-1$ and $\nu\in\Z_+$
\begin{align}\label{expansj}
        f(t) & = f_0t^{\nu-1} + O(t^{\nu+1}),\nonumber \\
        f'(t) & = \frac{\beta^3(1-2h_0)+1-\nu}{4\nu r_0\beta^2}f_0t^{\nu}+O(t^{\nu+2}),\\
        g(t) &=2\nu-1+O(t^2),\nonumber \\
        h(t)&=h_0+O(t^2).\nonumber
\end{align}
Given $Z_0= (f_0,h_0)\in\R^2$, let $Z(t,Z_0)$ be a solution on $P_j$ of \eqref{eq:G2ode11} with initial conditions determined by $Z_0$ at $t= 0$. Following the results of Amann in Chapter 9 of \cite{A80}, we consider solutions with initial conditions determined by $(f_0,h_0)$ in a small open neighbourhood of the origin. Corollary 9.3 of \cite{A80} tells us that $Z$ is at least continuously differentiable with respect to the initial condition $Z_0$. The following lemma shows that under the flow, this open neighbourhood is mapped to an open set that transversely intersects the stable manifold of the fixed point $z_0$. In particular, the open neighbourhood contains a curve of initial conditions which is mapped to a subset of the intersection $W^u(z_0)\cap W^s_T(z_0)$. Recall that for each $\tau_0$, we have a stable manifold $W^s_{\tau_0}(z_{\pm})$ of initial conditions at $\tau=\tau_0$ which flow towards the fixed point as $\tau\to\infty$.

\begin{lemma}
 Let $N(t,w) :=D_2Z(t,0)w$ for $w\in\R^2$. For large times $T$, the image of the linear map $w\mapsto N(T,w)$ is transverse to the tangent space of $W^s(z_0)$. 
\end{lemma}
\begin{proof}
We show that $N(t,w)$ is orthogonal to both of the stable eigenvectors for the fixed point $z_0$, namely
\[ (0,0,1,0) \hspace{1cm} \text{ and } \hspace{1cm} (1,-1,0,0).\] Firstly, since $g$ is independent of $h$, we see that $N(t,(0,1))$ is proportional to the vector $(0,0,0,1)$ for all $t$, and hence is orthogonal to $W^s(z_0)$. Next, we consider $w=(1,0)$; we claim that $N(t,w)$ lies in the region \[\mathcal{R}=\{(f,f',g,h) :f\geq f',f\geq -2f'\}\cup\{(f,f',g,h) :f\leq f',f\leq -2f'\}.\] In fact, we now show that this condition is preserved by the linearised flow, hence it is enough to prove the claim for some $t>0$. 

First consider the boundary $f+2f'=0$ of $\mathcal{R}$ given by the vector $(2,-1)$ in the $(f,f')$-plane. Let $B$ be the submatrix of $A$ given by the first two columns and rows. Then applying $B$ to $(2,-1)$ gives the direction of the flow at any time $t$ at that point. We apply the cross product and take the $z$-component, i.e.
\[\left(\left(\begin{array}{c} 2 \\ -1\end{array}\right)\times B\left(\begin{array}{c} 2 \\ -1\end{array}\right)\right)_z = 3\widetilde{\Phi}_1+\frac{1}{2}g(8\widetilde{\chi}_1-5(\widetilde{\Phi}_1+\widetilde{\Psi}_1)). \]
We have that $\Phi_1>0$, $g>0$ and $8\chi_1-5(\Phi_1+\Psi_1) = (b+r_0^3)^2(8a-5(b-r_0^3))$. By the conditions on $a$ and $b$ outlined in Remark \ref{rem:abcond}, the metric functions $a,b$ satisfy $ka-b+r_0^3>0$ for $k\in(1,2)$, hence the $z$-component of this vector is positive. The right-hand rule for the cross product implies that the flow must be in the direction indicated in Figure \ref{fig:flow}.

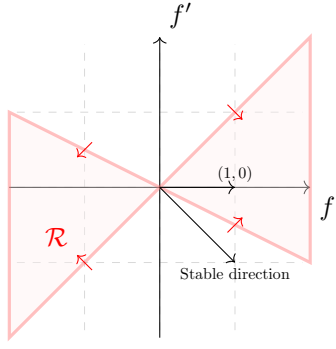
\begin{figure}[t]
\begin{center}
\begin{tikzpicture}[]
\draw[help lines, color=gray!30, dashed] (-1.9,-1.9) grid (1.9,1.9);
\draw[->] (-2,0)--(2,0) node[below right]{$f$};
\draw[->] (0,-2)--(0,2) node[above right]{$f'$};

\filldraw[color=red!60, fill=red!5, very thick,opacity=0.4] (0,0) -- (2,2) -- (2,-1) --(0,0);

\filldraw[color=red!60, fill=red!5, very thick,opacity=0.4] (0,0) -- (-2,-2) -- (-2,1) --(0,0);

\draw[->] (0,0)--(1,-1) node[below,scale=0.6] {Stable direction};
\draw[->] (0,0)--(1,0) node[above,scale=0.6] {$(1,0)$};

\draw[->,red] (0.9,1.1)--(1.1,0.9);
\draw[->,red] (0.9,-0.6)--(1.1,-0.4);
\draw[->,red] (-0.9,-1.1)--(-1.1,-0.9) node[red, above left] {$\mathcal{R}$};
\draw[->,red] (-0.9,0.6)--(-1.1,0.4);

\end{tikzpicture}
\end{center}
\caption{The directions of flow of \eqref{eq:G2ode11} along the lines $f'=f$ and $2f'=-f$.}
\label{fig:flow}
\end{figure}

On the other hand, taking the vector $(1,1)$ gives
\[\left(\left(\begin{array}{c} 1 \\ 1\end{array}\right)\times B\left(\begin{array}{c} 1 \\ 1\end{array}\right)\right)_z = -g(\widetilde{\Phi}_1+\widetilde{\Psi}_1+2\widetilde{\chi}_1)<0\]
since $g>0$  and $\Phi_1+\Psi_1+2\chi_1=(b+r_0^3)^2(b-r_0^3+2a)>0$ by the conditions on $a$ and $b$ given in Remark \ref{rem:abcond}. Hence the flow must be in the direction indicated in Figure \ref{fig:flow}. We have proved the claim that if $N(t,(1,0))$ lies in the region $\mathcal{R}$ for some $t>0$, then it does for all time.

In fact, for any $Z_0$ with $h_0=0$ and $f_0$ sufficiently small and for small $t>0$, we have \[N(t,(1,0))=\tfrac{\partial}{\partial f_0}Z(t,(f_0,0))=t^{\nu-1}(1,0,0,0)+O(t^\nu).\] So there exists $\epsilon>0$ such that $N(\epsilon,(1,0))\in\mathcal{R}$, and hence $N(t,(1,0))\in\mathcal{R}$ for all time. We have proved that $N(T,(1,0))$ is orthogonal to $W^s(z_0)$ for large $T$ and so the image of the linear map $w\mapsto N(T,w)$ is transverse to the stable manifold for large times $T$.
\end{proof}

In particular, we can choose initial conditions which flow to the 1-dimensional intersection of the transversal with the stable manifolds of the fixed points $z_{\pm}$ and hence obtain a 1-parameter family of global solutions. Each member of the 1-parameter family of solutions with limit $z_+$ is gauge equivalent to a corresponding member of the family tending to $z_-$ under the $\G$-invariant gauge transformation of Remark \ref{rem:gaugetrans}. Hence there is a 1-parameter family of solutions on each bundle up to gauge transformation given by this construction.

\begin{lemma}
We can parameterise this family of solutions by $f_0$.
\end{lemma}
\begin{proof}
Indeed, $N(T,(0,1))$ is proportional to $v=(0,0,0,1)$ and $v$ is invariant under both the linearisation $A(t)$ and the flow. However, $v\not\in T_{z_{\pm}}W^s(z_{\pm})$ and hence $N(T,(0,1))$ is not tangent to $W^s(z_{\pm})$. So in a small neighbourhood of the origin in the $(f_0,h_0)$-plane, the projection to the $f_0$-axis of the curve of initial conditions corresponding to this family of perturbations is injective. 
\end{proof}

To summarise, we have shown that a 2-dimensional perturbation of the initial condition of the abelian solution $\Aab$ flows to a manifold transverse to the stable manifold of the fixed point $z_0$. By applying the Inclincation Lemma, we see that this transversal manifold then flows to become $C^1$-close to the unstable manifold of $z_0$ and hence lies in the intersection with the stable manifold of $z_{\pm}$. Since the non-autonomous part of the system is sufficiently small at large enough times, we can see that a curve of solutions in this intersection flows into the fixed points $z_{\pm}$ as $t\to\infty$. This curve of solutions exists for all $t\geq0$ for $f_0$ in some interval $(-\epsilon, \epsilon)$, giving a 1-parameter family of global solutions. This completes the proof of Theorem \ref{thm:main}.

\subsection{Asymptotics of the solutions}

The fixed point $z_+$ defines a connection on the nearly K\"ahler $(S^3\times S^3)/\mathbb{Z}_4$, namely
\[A_{\infty}=\frac{1}{3}\sum_{i=1}^3 E_i\otimes (e_i+e_i').\]
We can consider $S^3\times S^3$ as the homogeneous manifold $G/K=\SU(2)^3/\Delta\SU(2)$ and let $e_i''$ be a coframe on the third $\SU(2)$-factor. The nearly K\"ahler metric is induced from a multiple of the Cartan-Killing form
\[B(X,Y)=\text{Tr}_{\mathfrak{g}}(\text{ad}(X)\text{ad}(Y)), \hspace{1cm} \forall X,Y\in\mathfrak{g}.\]
The subspace $\mathfrak{k}$ of $\mathfrak{g}$ is generated by $E_i+E_i'+E_i''$ for $i=1,2,3$. Let $\mathfrak{m}$ be an orthogonal complement to $\mathfrak{k}$ with respect to $B$, then the canonical invariant connection on $G/K$ is given by
\[A=\frac{1}{3}\sum_{i=1}^3 (E_i+E_i'+E_i'')\otimes (e_i+e_i'+e_i'')\]
on the $\SU(2)$-bundle $P_{\lambda}$ over $G/K$, where $\lambda:\Delta\SU(2)\to\SU(2)$. Pulling $A$ back to $\SU(2)^2$ via the inclusion map $\SU(2)^2\to\SU(2)^3$ yields the connection $A_{\infty}$. The tangent bundle of $G/K$ can be identified with the vector bundle associated to $G\to G/K$ via the representation
\[\rho_{\mathfrak{m}}:K\to \text{GL}(\mathfrak{m}).\]
Then $A$ induces a connection on $T(G/K)$ which has holonomy contained in $\SU(3)$ and skew-symmetric torsion. Theorem 10.1 of \cite{FI01} tells us that such a connection is unique, and hence it must coincide with the so-called canonical Hermitian connection. This connection is given by the formula
\[\nabla -\frac{1}{2}J(\nabla J)\]
where $\nabla$ is the Levi-Civita connection and $J$ is the almost complex structure on $G/K$. In summary, the limiting connection of the solutions in Theorem \ref{thm:main} is the canonical connection on the nearly K\"ahler manifold $(S^3\times S^3)/\mathbb{Z}_4$.

\appendix
\section{Extending \texorpdfstring{$\SU(2)^2$}{SU(2)xSU(2)}-invariant connections over the singular orbit}
\label{sec:invtens}

For the $G_2$-manifold $M_{m,n}$, the singular orbit is $Q=S^3\times S^3/K_{m,n}$; we now determine the conditions for an invariant connection to extend smoothly over $Q$. This singular orbit is a circle bundle over $D=S^2\times S^2=\SU(2)^2/T^2$, i.e. $Q=\SU(2)^2\times_{T^2} S^1$, and $M_{m,n}$ is the total space of a $\C\times S^1$-bundle over $D$, i.e. $M_{m,n}=\SU(2)^2\times_{T^2}(S^1\times\C)$. Here, $T^2$ acts on $S^1$ with weight $(m,n)$, and on $\C$ with weight $(2,-2)$. Since $m,n$ are coprime, $K_{m,n}\cong\ $U(1) is embedded in $T^2$ by $e^{i\theta}\mapsto (e^{in\theta},e^{-im\theta})$.

Define $\mathfrak{p}$ by $\su(2)\oplus\su(2)=\text{Lie}(K_{m,n})\oplus\mathfrak{p}$, then given a point $q \in Q$, we can identify $T_qQ$ with $\mathfrak{p}$ and $T_qM$ with $\mathfrak{p}\oplus V$. As a $K_{m,n}$-representation, $\mathfrak{p}$ is induced by the $T^2$-representation $\R\oplus\C_{2,0}\oplus\C_{0,2}$, and so as a real representation, $\mathfrak{p}=\R\oplus\R^2_{2n}\oplus\R^2_{2m}$. Let $t,x$ be coordinates on $V\cong\R^2$, then we see from \cite{FHN18} that the coframe 
\[se_3-re_3' = \frac{1}{(m+n)t}dx;\]
also, recall that $e_1,e_2$ and $e_1',e_2'$ are coframes on $\R^2_{2n}$ and $\R^2_{2m}$ respectively, while $me_3+ne_3'$ generates the trivial $\R\subset\mathfrak{p}$.

Eschenburg-Wang \cite{EW00} provide a method for finding the conditions to extend tensors smoothly over the singular orbit. A tubular neighbourhood of $Q$ in $M_{m,n}$ is equivariantly diffeomorphic to a neighbourhood of the zero section in the vector bundle $\SU(2)^2\times_{K_{m,n}} V\to Q$, where $V$ is the irreducible 2-dimensional real representation of $K_{m,n}$ with weight $2(m+n)$. In other words, the tubular neighbourhood is modelled on $W=V\oplus\mathfrak{p}=\R^2_{2(m+n)}\oplus\R\oplus\R^2_{2n}\oplus\R^2_{2m}$. By identifying $\R^2\cong\C$, the action of $\zeta\in K_{m,n}\cong\ $U(1) on $(z_1,y,z_2,z_3)\in W$ is given by
\[\zeta\cdot(z_1,y,z_2,z_3) = (\zeta^{2(m+n)}z_1,y,\zeta^{2n}z_2,\zeta^{-2m}z_3).\]

Let $P_{j}$ be a principal $H$-bundle on $M_{m,n}$, parameterised by isotropy homomorphisms $\lambda_j:K_{m,n}\to H$. Then $K_{m,n}$ acts on $\mathfrak{h}$ by $\Ad\circ\lambda_j$. We consider $\SU(2)^2$-invariant elements of $\Omega^1(\Ad P_j)$. The invariance property means that such a 1-form is determined by its restriction to the fibre of $\SU(2)^2\times_{K_{m,n}}V$ over the point $q\in Q$. Hence, it is given by a $K_{m,n}$-equivariant map $L:V\to W^*\otimes\mathfrak{h}$. Fix $v_0=1\in S^1\subset V$, then by $K_{m,n}$-equivariance, $L$ is uniquely determined by the curve $t\mapsto L(tv_0)$ whose image lies in the subspace of $K_0$-invariant sections of $\Omega^1(\Ad P_j)$. Indeed, for each $t$, the restriction of such a 1-form to the corresponding orbit can be written as a map from the tangent space at a single point to the fibre over that point of the adjoint bundle, which is a copy of $\mathfrak{h}$. This is the map $\Lambda$ given in \eqref{eq:Lambda}. The equivariance condition on $\Lambda$ with respect to each isotropy subgroup is exactly the condition that the corresponding section has values in the adjoint bundle.

To determine which invariant connections on $P_j$ extend smoothly over $Q$, we want to consider the evaluation at $v_0$ of homogeneous $K_{m,n}$-invariant polynomials $\phi:\R^2_{2(m+n)}\to\Hom(W,\mathfrak{h})\cong W^*\otimes\mathfrak{h}$ of minimal degree. Then Lemma 1.1 of \cite{EW00} implies that the corresponding  connection extends over the singular orbit $Q$ if and only if the Taylor series expansion around zero of the coefficient of each component of the representation $W^*\otimes\mathfrak{h}$ has the same parity as the corresponding homogeneous equivariant polynomial $\phi$ and the first non-zero term has the same degree as $\phi$. We now see what these conditions mean practically for connections on bundles with gauge group $H=\SU(2)$.

The action of $K_{m,n}$ on $\mathfrak{h}=\su(2)$ is given by $\Ad\circ\lambda_j$ on the bundle $P_j$. We use the matrix basis of $\su(2)$ given by
\[E_1=\left(\begin{array}{cc}
     0& i \\ i& 0
\end{array}\right),\hspace{1cm} E_2 = \left(\begin{array}{cc}
     0& -1 \\ 1& 0
\end{array}\right),\hspace{1cm} E_3=\left(\begin{array}{cc}
     i& 0 \\ 0& -i
\end{array}\right).\]
In this case, irreducible invariant connections are determined by a morphism of $\Z_{2(m+n)}$-representations $\Lambda:(\R\oplus\C_{2n})\oplus(\R\oplus\C_{-2m})\to\R\oplus\C_{2j}$ only if $j\equiv n\mod2(m+n)$ or $j\equiv -m \mod2(m+n)$. Since $\C_{2n}=\C_{-2m}$ as $\Z_{2(m+n)}$-representations, there are no restrictions on $\Lambda$ for these choices of $j$.

\begin{lemma}
\label{lemma:nonabmn}
Consider the manifold $M_{m,n}$. Then the invariant connection over the principal orbits
\[A  = A_{12}(E_1\otimes e_1+E_2\otimes e_2)+ A_{12}'(E_1\otimes e'_1+E_2\otimes e'_2)+A_{rs}(E_3\otimes (se_3-re_3'))+A_{mn}(E_3\otimes(me_3+ne_3'))\]
on $P_j$ with $j\equiv n\mod2(m+n)$ or $j\equiv -m \mod2(m+n)$ extends to a smooth section of $\Omega^1(\Ad P_j)$ if and only if $A_{rs},A_{mn}$ are even, with $A_{rs}(0)=0$, and we have
\[\begin{cases}A_{12} \text{ even and } O(t^{\left\vert\frac{j-n}{m+n}\right\vert}), A_{12}' \text{ odd and } O(t^{\left\vert\frac{j+m}{m+n}\right\vert}) & \text{ for }j\equiv n \mod 2(m+n); \\
A_{12} \text{ odd and } O(t^{\left\vert\frac{j-n}{m+n}\right\vert}), A_{12}' \text{ even and } O(t^{\left\vert\frac{j+m}{m+n}\right\vert}) & \text{ for }j\equiv -m \mod 2(m+n).
\end{cases}\]
\end{lemma}
\begin{proof}
Firstly, note that $(\Ad\circ\lambda_j)(e^{i\theta})E_3=E_3$. Then, considering the subspace $V\subset W$, we can define an equivariant polynomial $V\to \Hom(V,\su(2))$ as
\[\phi_{\ell}(x)(z_1)=\langle\ell x,z_1\rangle E_3 \text{ for } \ell\in\{1,i\}.\]
These are equivariant homogeneous degree 1 polynomials which correspond upon evaluation at $v_0=1$ to the 1-forms $dt$ for $\ell=1$ and $dx$ for $\ell=i$. Then, the coefficient of $dx=(m+n)t(se_3-re_3')$ must be odd and be of order $O(t)$. Hence, the coefficient $A_{rs}$ of $se_3-re_3'$ must be even and of order $O(t^2)$.

On the other hand, for the trivial part of $\mathfrak{p}$ which is generated by $me_3+ne_3'$, we take the constant polynomial which maps to the identity in $\Hom(V,\su(2))$. Then since the action of $K_{m,n}$ is trivial, the equivariance condition is satisfied. Thus the coefficient $A_{mn}$ of $me_3+ne_3'$ must be even.

Consider the bundle $P_j$ with positive $j\equiv n,-m\mod2(m+n)$, then the representation $(\R^2_{2n})^*\otimes\su(2)$ of $K_{m,n}$ has weight $2(j-n)$. The space of homogeneous polynomials of degree $d$ on $V$ is isomorphic to $\Sym^dV$, and $K_{m,n}$ acts on $\Sym^dV$ with weight $2d(m+n)$. Then the smallest $d$ such that the corresponding degree $d$ polynomial is equivariant is $d=\frac{j-n}{m+n}$. Similarly, the smallest $d$ such that the corresponding degree $d$ polynomial with values in $(\R^2_{-2m})^*\otimes\su(2)$ is equivariant is $d=\frac{j+m}{m+n}$.

In conclusion, if $j\geq0$, then on $P_j$, we have that $A_{12}$ is $O(t^{\frac{j-n}{m+n}})$ and has the same parity as $\frac{j-n}{m+n}$, while $A_{12}'$ is $O(t^{\frac{j+m}{m+n}})$ and has the same parity as $\frac{j+m}{m+n}$.

Finally, if $j<0$, a similar calculation shows that on $P_j$, $A_{12}$ is $O(t^{\frac{-j+n}{m+n}})$ and the same parity as $\frac{-j+n}{m+n}$, while $A_{12}'$ is $O(t^{\frac{-j-m}{m+n}})$ and the same parity as $\frac{-j-m}{m+n}$.
\end{proof}

\section{Additional Dynamical Systems Theory}\label{sec:dynsys}

The material in this section can be found in \cite{T12}; it is designed to give the necessary background theory in dynamical systems used in the proof of Theorem \ref{thm:main}. More precisely, we set out the conditions for which a perturbation of an autonomous system by a non-autonomous term still yields a stable manifold.

Consider the non-autonomous system
\begin{equation}\label{eq:nonautogen}\dot{x}(t)=A(t)x(t)+g(t).\end{equation}
We can produce an ansatz for the solution of such an equation, known as a Variation of Constants, which takes the form
\[x(t) = \Pi(t,t_0)c(t), \hspace{1cm} c(t_0)=x(t_0)=x_0.\]
Here, $\Pi$ is called the principal matrix solution of the system
\[\begin{cases}\dot{\Pi}(t,t_0)=A(t)\Pi(t,t_0) \\ \Pi(t_0,t_0)=\text{Id}.\end{cases}\]
Substituting this ansatz solution yields the integral equation
\[c(t)=x_0+\int_{t_0}^t\Pi(t_0,s)g(s)ds\]
and hence we have the following theorem.

\begin{theorem}\label{thm:varofconst}
The solution of \eqref{eq:nonautogen} corresponding to the initial condition $x(t_0)=x_0$ is 
\[x(t)=\Pi(t,t_0)x_0+\int_{t_0}^t\Pi(t,s)g(s)ds.\]
\end{theorem}

Now we consider the autonomous system
\[\dot{x} = A(x-x_0) +g(x)\]
where $g(x)$ is the non-linear part. Assume $x_0=0$ is a hyperbolic fixed point. Then we can apply Theorem \ref{thm:varofconst} to rewrite the equation as
\[x(t) = e^{tA}x(0)+\int_0^te^{(t-r)A}g(x(r))dr.\]

Denote by $x_{\pm}$ the projection $P^{\pm}$ of $x(0)$ onto the stable and unstable subspaces. We want to understand the condition on $x(0)=x_++x_-$ for which $x(t)$ remains bounded for all time. Let $P(t)$ be the function defined by $P^+$ for $t>0$ and $-P^-$ for $t\leq0$. Then Teschl writes $x_-$ as a function of $x_+$ and hence rewrites $x(t)$ as $x(t)=K(x)(t)$ where
\[K(x)(t)=e^{tA}x_++\int_0^{\infty}e^{(t-r)A}P(t-r)g(x(r))dr.\]

We now specialise to the equation of interest \eqref{eq:defF}, namely
\[\frac{dz}{d\tau}=F(z)+G(z,\tau).\]
Let $t_0$ be the initial time and temporarily set $G=0$. Then we can write $z(t)=K(z)(t)$ where
\[K(z)(t)=e^{(t-t_0)DF(0)}z_++\int_{t_0}^{\infty}e^{(t-r)DF(0)}P(t-r)(F(z(r))-DF(0)z(r))dr.\]
With this setup, we can apply Theorem 9.3 of \cite{T12}. The result is that for all $t_0$, and all $z_+$ in a small enough neighbourhood of 0 within the stable eigenspace, $K$ is a contraction in the space of bounded functions $C^k_0([t_0,\infty),\R^4)$  and has a fixed point $\psi(t,z_+)$ by the contraction principle. Thus the stable manifold is given as $z_++P^-\psi(t_0,z_+)$.

We now adapt the proof of Theorem 9.3 of \cite{T12} in the case where $G$ is non-zero. Then the function $K(z)$ becomes 
\[K_G(z)(t)=e^{(t-t_0)DF(0)}z_++\int_{t_0}^{\infty}e^{(t-r)DF(0)}P(t-r)(F(z(r))-DF(0)z(r)+G(z(r),r))dr.\]

By \eqref{eq:Glim}, for $t_0$ large enough, we have 
\begin{equation}\label{eq:Gbound}\vert G(z(r),r)\vert\leq ce^{-t_0}\end{equation}
can be made arbitrarily small and $G$ is locally Lipschitz. Here, $c$ is independent of $z$ near $z_+$ by \eqref{eq:DGlim}.  Then, 
\begin{align*}
    \Vert K_G(x)-K_G(y)\Vert &= \sup_{t\geq0}\left\vert \int_{t_0}^{\infty}e^{(t-r)DF(0)}P(t-r)(F(x)-F(y)-DF(0)(x-y)+(G(x,r)-G(y,r)))dr\right\vert\\
    &\leq C\sup_{t\geq0}\int_{t_0}^{\infty}e^{-\alpha_0\vert t-r\vert}(\vert (F-DF(0))(x-y)\vert +\vert G(x,r)-G(y,r)\vert)dr
\end{align*}
where $0<\alpha_0<\min\{\vert \textup{Re}(\alpha)\vert:\alpha\in\sigma(DF(0))\}$, where $\sigma(DF(0))$ is the set of eigenvalues of $DF(0)$. Then, since the Jacobian of $F-DF(0)$ vanishes at $0$, we have
$\vert (F-DF(0))(x-y)\vert \leq \epsilon\vert x-y\vert$; together with \eqref{eq:Gbound}, we see that $K_G$ is a contraction for a large enough choice of $t_0$. By the contraction principle, we again have a fixed point which defines the stable manifold at the fixed point.

Hence, a small perturbation of an autonomous system by a non-autonomous term does not affect the existence of a stable manifold at the fixed point.

\bibliographystyle{plain}
\bibliography{refs.bib}

\end{document}